\documentclass[a4paper,11pt,reqno]{amsart}

\usepackage{amsfonts}
\usepackage{amssymb}
\usepackage{amscd}
\usepackage{amsthm}
\usepackage{a4wide}
\usepackage{mathrsfs}
\usepackage[applemac]{inputenc}
\usepackage{amsmath}
\usepackage{amssymb}
\usepackage{amsthm}
\usepackage{textcomp}
\usepackage{graphicx}
\usepackage{enumerate}
\usepackage{mathrsfs}
\usepackage{frcursive}
\usepackage{tikz}
\usepackage[cyr]{aeguill}
\usepackage{xspace}
\usepackage{hyperref}
\usepackage{appendix}

\newtheorem{defn}{Definition}[section]
\newtheorem{lemma}[defn]{Lemma}
\newtheorem{prop}[defn]{Proposition}
\newtheorem{theo}[defn]{Theorem}
\newtheorem{coro}[defn]{Corollary}

\newtheorem{claim}{Claim}
\newtheorem{rk}[defn]{Remark}

\def\Ric{\mathop{\rm Ric}\nolimits}

\def\Rm{\mathop{\rm Rm}\nolimits}
\def\tr{\mathop{\rm tr}\nolimits}

\def\vol{\mathop{\rm vol}\nolimits}

\def\vol{\mathop{\rm Vol}\nolimits}

\def\div{\mathop{\rm div}\nolimits}

\def\Id{\mathop{\rm Id}\nolimits}

\def\Li{\mathop{\rm \mathscr{L}}\nolimits}
\def\Hyp{\mathop{\rm (\mathscr{H})}\nolimits}

\def\Cod{\mathop{\rm Cod}\nolimits}
\def\Ric{\mathop{\rm Ric}\nolimits}

\def\Rm{\mathop{\rm Rm}\nolimits}
\def\tr{\mathop{\rm tr}\nolimits}

\def\vol{\mathop{\rm vol}\nolimits}

\def\vol{\mathop{\rm Vol}\nolimits}

\def\div{\mathop{\rm div}\nolimits}

\def\Id{\mathop{\rm Id}\nolimits}

\def\Li{\mathop{\rm \mathscr{L}}\nolimits}
\def\Hyp{\mathop{\rm (\mathscr{H})}\nolimits}

\def\Hyp{\mathop{\rm \mathscr{H}}\nolimits}

\newcommand*{\avi}{\mathop{\ooalign{$\int$\cr$-$}}}

\title[Stability of non compact steady and expanding gradient Ricci solitons]{Stability of non compact steady and expanding gradient Ricci solitons}

\author[Alix Deruelle]{Alix Deruelle}
\address[Alix Deruelle]{Mathematics Institute, University of Warwick, Gibbet Hill Rd, Coventry, West Midlands CV4 7AL}
\email{A.Deruelle@warwick.ac.uk}

\begin{document}
\begin{abstract}
We study the stability of non compact steady and expanding gradient Ricci solitons. We first show that linear stability implies dynamical stability. Then we give various sufficient geometric conditions ensuring the linear stability of such gradient Ricci solitons.
\end{abstract}
\maketitle

\section{Introduction}

Once a smooth manifold $M$ without boundary is fixed, the Ricci flow, introduced by Hamilton, can be considered as a dynamical system on the space of metrics of $M$ modulo the action of diffeomorphisms and homotheties. The fixed points of such dynamical system are called \textit{Ricci solitons}. By definition, these solutions look like $\tau(t)\phi_t^*g$ where $\tau(t)$ is a positive scaling factor, where $(\phi_t)_t$ is a one parameter family of diffeomorphisms and $g$ is a fixed metric. If this family of diffeomorphisms is generated by a gradient vector field, we call these fixed points \textit{gradient} Ricci solitons. Finally, if one plugs this family of metrics into the Ricci flow equation and evaluates this expression at a fixed time, one gets a static equation.\\

 More precisely, a gradient Ricci soliton is a triplet $(M,g,\nabla f)$ where $(M,g)$ is a Riemannian manifold and $f:M\rightarrow\mathbb{R}$ is a smooth function (called the potential function) satisfying
\begin{eqnarray*}
\frac{1}{2}\Li_{\nabla f}(g)=\Ric(g)+\frac{\epsilon g}{2},
\end{eqnarray*}
where $\epsilon\in\{-1,0,1\}$. We will focus on the steady ($\epsilon=0$) and the expanding ($\epsilon=1$) cases.
We assume that $(M,g)$ is complete, this suffices to ensure the completeness of the vector field $\nabla f$ : \cite{Zha-Zhu}.

Let us fix a gradient Ricci soliton $(M,g_0,\nabla^{g_0}f_0)$. Define $X^0:=\nabla^{g_0}f_0$.
In this paper, we study the stability of such gradient Ricci soliton under Ricci flow, i.e.
\[
\left\{
\begin{array}{rl}
&\partial_tg=-2\Ric(g(t)) \quad\mbox{on}\quad M\times (0,+\infty),\\
& g(0)=g_0+h.
\end{array}
\right.
\]

As we want to prove convergence to the fixed background metric $g_0$, it is more convenient to consider the following modified ($\epsilon$-)Ricci flow (MRF) :
\[
\left\{
\begin{array}{rl}
&\partial_tg=-2\Ric(g(t))-\epsilon g(t)+\Li_{X^0}(g(t)) \quad\mbox{on}\quad M\times (0,+\infty),\\
& g(0)=g_0+h.
\end{array}
\right.
\]
As a first remark, the (MRF) flow is equivalent with the usual Ricci flow : see lemma \ref{equ-flo}. Secondly, there is a vast amount of literature concerning the stability of Einstein metrics (constant potential function) under the Ricci flow both on compact and non compact manifolds. Rather than making an exhaustive list of the existing work, we prefer to discuss the differences between our approach and the other approaches shared by different works. For instance, in the compact Ricci-flat case \cite{Has-Mul-Lam} and the shrinking  gradient Ricci soliton case \cite{Ach-Shr}, \cite{Kro-Sta-Ins}, the authors consider the previous modified Ricci flow with $\nabla^{g(t)} f(t)$ instead of the fixed vector field $\nabla^0f_0$ where $f(t)$ is a smooth function depending on time implicitly produced by a minimizing procedure involving the relevant functional (Perelman's energy or entropy). As we consider non Einstein steady and expanding Ricci solitons, they are necessarily non compact : in that setting, the corresponding functionals are not even defined. \\

For that reason, we decide to follow the strategy of \cite{Sch-Sch-Sim} on the stability of non compact Hyperbolic spaces. In the spirit of the so called DeTurck's trick, as the (MRF) is a degenerate parabolic equation,  we first need to consider the following modified ($\epsilon$-)Ricci harmonic map heat flow (MRHF) : 
\[
\left\{
\begin{array}{rl}
&\partial_tg=-2\Ric(g(t))-\epsilon g(t)+\Li_{X^0}(g(t))+\Li_{V}(g(t)) \quad\mbox{on}\quad M\times (0,+\infty)\\
& g(0)=g_0+h
\end{array}
\right.
\]
where $V$ is a vector field defined in coordinates by : $V_i:=g_{ik}(\Gamma(g)^k_{rs}-\Gamma(g_0)_{rs}^k)g^{rs},$ or  globally,
\begin{eqnarray}\label{de-turck-vect}
V(g):=\div_{g_0}g-\frac{1}{2}\nabla^0\tr_{g_0}g.
\end{eqnarray}

As in \cite{Sch-Sch-Sim}, we consider perturbations that are close in the $L^{\infty}$ norm to the fixed background gradient Ricci soliton.
In order to state our first result, we need (to recall) several definitions.  

\begin{defn}
\begin{itemize}
\item The weighted laplacian operator associated to a gradient Ricci soliton  $(M,g,\nabla f)$ is $\Delta_f:=\Delta+\nabla f$.\\
\item The weighted Lichnerowicz operator associated to a gradient Ricci soliton $(M,g,\nabla f)$ is $L:=-\Delta_f-2\Rm(g)\ast$, where $(\Rm(g)\ast h)_{ij}:=\Rm(g)_{iklj}h_{kl}$, for a symmetric $2$-tensor $h$.\\
\item A gradient Ricci soliton  $(M,g,\nabla f)$ is strictly stable if there exists a positive $\lambda$ such that
 \begin{eqnarray*}
\int_M\langle Lh,h\rangle d\mu_f=\int_{M}\arrowvert \nabla h\arrowvert^2-\langle2\Rm(g)\ast h,h\rangle d\mu_f\geq\lambda\int_{M}\arrowvert h\arrowvert^2d\mu_f,
\end{eqnarray*}
for any symmetric $2$-tensors $h$ with compact support, where $d\mu_f:=e^fd\mu(g)$, $d\mu(g)$ being the Riemannian measure associated to the metric $g$. \\
\end{itemize}
\end{defn}
This notion of stability is legitimated by the fact that the linearization of the (MRHF) is exactly the Lichnerowicz operator.
Finally, we will consider the following convexity assumptions on the geometry of the fixed background gradient Ricci soliton (called under the same name $(\Hyp)$) :
\begin{itemize}
\item If $(M,g,\nabla f)$ is a steady gradient Ricci soliton then $(\Hyp)$ consists in
 \begin{eqnarray*}
 \Ric\geq 0,\quad\lim_{+\infty}R=0\quad\sup_M\arrowvert\Rm(g)\arrowvert<+\infty.
\end{eqnarray*}

 \item If $(M,g,\nabla f)$ is an expanding gradient Ricci soliton then $(\Hyp)$ consists in
  \begin{eqnarray*}
\sup_M\arrowvert\Rm(g)\arrowvert<+\infty, \quad\liminf_{x\rightarrow+\infty}\frac{f(x)}{r_p(x)^2}>0.
\end{eqnarray*}
 \end{itemize}
 
Our main theorem is

\begin{theo}\label{theo-MRHF}
Let $(M,g_0,\nabla^0f_0)$ be a strictly stable expanding or steady gradient Ricci soliton satisfying $(\Hyp)$. Then it is dynamically stable.

More precisely, if $g$ is a smooth metric such that
\begin{eqnarray*}
\|g-g_0\|_{L_{f_0}^2}:=\left(\int_M\arrowvert g-g_0\arrowvert^2_{g_0}d\mu_{f_0}\right)^{1/2}=:I
\end{eqnarray*}
 is finite, there exists $\epsilon:=\epsilon(n,k_0,\lambda,I)>0$ such that the following holds, where $k_0:=\sup_M\arrowvert\Rm(g_0)\arrowvert$. If $g$ satisfies $\|g-g_0\|_{\infty}\leq \epsilon$, then there exists an immortal solution to the Ricci harmonic map heat flow converging exponentially fast to $g_0$ in any $C^k$ norm for $k\in \mathbb{N}$.
\end{theo}

Then, we are able to go back to the (MRF) : 

\begin{theo}\label{theo-MRF}
Let $(g(t))_{t\in[0,+\infty)}$ be the solution to (MRHF) obtained in theorem \ref{theo-MRHF}. Then there exists a one parameter family of diffeomorphisms $(\psi_t)_{t\in[0,+\infty)}$ of $M$ that satisfies $\psi_0=Id_M$ and such that $(\psi_t^*g(t))_{t\in[0,+\infty)}$ is a smooth solution to (MRF) with the same initial condition. Moreover, there exists a diffeomorphism $\psi_{\infty}$ of $M$ such that 
\begin{eqnarray*}
\psi_t\rightarrow \psi_{\infty},\quad \psi_t^*g(t)\rightarrow \psi_{\infty}^*g_0,
\end{eqnarray*}
as $t$ tends to $+\infty$, where the convergence is in (global) $C^k$ norms, for any $k\in\mathbb{N}$.\\
\end{theo}

\begin{rk}
If one is interested in weighted $L^p_{f_0}$ estimates with $p\geq 2$, one can actually prove theorem \ref{theo-MRHF} in this setting in the case of gradient steady Ricci solitons.
In fact, such approach has been investigated by Bamler \cite{Bam-Sym} for symmetric spaces : the main tools coming from harmonic analysis are the Hardy-Littlewood maximal local operator together with the Hardy-Littlewood maximal inequality which holds for weights as in the case of gradient steady Ricci solitons, since the potential function has bounded mean oscillation. Nonetheless, this approach is not applicable in the case of expanding gradient Ricci solitons because the potential function has not bounded mean oscillation...
\end{rk}

In a second part, we investigate sufficient geometric assumptions that imply $\lambda$-stability, the assumption $(\Hyp)$ being quite reasonable to assume.

We consider positively curved gradient Ricci solitons first. This choice is actually made with respect to the existing examples. Before stating the results, we describe the strategy to prove $\lambda$-stability. This is nothing more than proving that the bottom of the spectrum of the Lichnerowicz operator acting on symmetric two tensors that are square integrable with respect to the weighted measure $d\mu_f$ is positive. It is well-known that the spectrum can be decomposed into two parts : the essential spectrum and the discrete spectrum. We first analyse the essential spectrum by studying the bottom of the spectrum of the weighted laplacian acting on square integrable (in the weighted sense) functions (propositions \ref{Mun-Wan-Har-Sgs} and \ref{Mun-Wan-Har-Egs}). Then we investigate the non existence of non trivial eigenfunctions of $L$ associated to a non positive eigenvalue by proving a priori exponential decay in the spirit of Agmon \cite{Agm-Dec} and Simon \cite{Sim-Dec} : theorems \ref{Agmon-steady} and \ref{Agmon-expanding}.

\begin{itemize}
\item Concerning gradient Ricci steady solitons, the only known one with positive curvature is the Bryant soliton in dimension greater than $2$. It behaves like a paraboloid. We are able to prove the following :

\begin{theo}\label{spec-sgs-pos-cur}
Let $(M,g,\nabla f)$ be a non flat steady gradient Ricci soliton with non negative curvature operator with 
\begin{eqnarray*}
\lim_{+\infty} R=0\quad;\quad \liminf_{t\rightarrow +\infty}e^{\alpha t}\inf_{f=t}\Ric>0,
\end{eqnarray*}
for some $\alpha\in[0,1)$.
 Then the bottom of the spectrum of $L$ is positive.
\end{theo}

\begin{rk}
Note that the curvature of the cigar soliton on $\mathbb{R}^2$, discovered by Hamilton, decays precisely as $e^{-f}$. Moreover, the author has shown that this decay is essentially sharp among nonnegatively curved steady gradient Ricci solitons \cite{Der-sgs}. In that sense, the previous theorem can be seen as a gap theorem in terms of curvature decay. Nonetheless, even if theorem \ref{spec-sgs-pos-cur} discards the cigar soliton, Ma and Witt have proved its stability \cite{Ma-Wit} using the special structure of the Ricci flow in dimension $2$.
\end{rk}
We observe that this curvature condition is satisfied by the Bryant soliton too since in this case, $\liminf_{t\rightarrow+\infty}t^2\inf_{f=t}\Ric>0$. Hence,

\begin{coro}\label{thm-2}
The Bryant soliton is dynamically stable. 
\end{coro}

\item Concerning expanding gradient Ricci solitons, Bryant, in unpublished notes, has also built a one parameter family of rotational symmetric examples on $\mathbb{R}^{n}$ for $n\geq 3$ [Section $5$, Chap. $1$, \cite{Cho-Lu-Ni-I}] asymptotic to the cone $(C(\mathbb{S}^{n-1}),dr^2+(cr)^2g_{\mathbb{S}^{n-1}})$ with $c\in(0,+\infty)$. These examples have positive curvature if and only if $c\in(0,1)$. In that setting, we are able to prove the following :

\begin{theo}\label{lic-op-dis-pos-curv}
Let $(M,g,\nabla f)$ be an expanding gradient Ricci soliton with non negative curvature operator. Then the bottom of the spectrum of $L$ is positive. In particular, it is dynamically stable.
\end{theo}

In dimension $3$, we are able to only assume positive scalar curvature thanks to an estimate due to Anderson and Chow \cite{And-Cho}. More precisely,

\begin{theo}\label{lic-op-dis-pos-scal-curv}
Let $(M,g,\nabla f)$ be a $3$-dimensional expanding gradient Ricci soliton satisfying $(\Hyp)$ with positive scalar curvature with $$\liminf_{x\rightarrow+\infty} e^{\alpha f(x)}R(x)>0,$$ for some $\alpha\in[0,1)$. Then the bottom of the spectrum of $L$ is positive. In particular, it is dynamically stable.
\end{theo}

\end{itemize}
\begin{rk}
Chih-Wei Chen and the author proved that all expanding gradient Ricci solitons with quadratic curvature decay are asymptotically conical : the asymptotic cone is a metric cone whose section is a smooth compact manifold endowed with a $C^{1,\alpha}$ metric. See \cite{Che-Der} for more details. For instance, the assumption of theorem \ref{lic-op-dis-pos-scal-curv} on the asymptotical behaviour of the scalar curvature is satisfied if the section of the asymptotic cone is a smooth positively curved sphere with curvature greater than $1$ : in that case, $\liminf_{x\rightarrow+\infty}r_p^2(x)R(x)$ is the minimum of the scalar curvature restricted to the section which is positive. 
 \end{rk}

\begin{rk}
According to Siepmann \cite{Sie-PHD}, the condition on the decay of the curvature in theorem \ref{lic-op-dis-pos-scal-curv} is essentially sharp. Indeed, Siepmann proved that the Ricci curvature decay as $e^{-\alpha f}$ for any $\alpha\in[0,1)$ in case the asymptotic cone of an expanding gradient Ricci soliton is Ricci flat.
\end{rk}

What about the stability of the Bryant expanding gradient Ricci solitons with negative curvature, i.e. the ones whose asymptotic cones are  $(C(\mathbb{S}^{n-1}),dr^2+(cr)^2g_{\mathbb{S}^{n-1}})$ with $c\in(1,+\infty)$ ?

Actually, it seems that the stability of expanders is intimately related to rigidity : Chodosh \cite{cho-EGS} proved that any positively curved expanding gradient Ricci solitons asymptotic to $(C(\mathbb{S}^{n-1}),dr^2+(cr)^2g_{\mathbb{S}^{n-1}})$ with $c\in(0,1)$ are rotationally symmetric and theorem \ref{lic-op-dis-pos-curv} implies their stability. The principle ingredient here is the maximum principle for symmetric $2$-tensors due to Hamilton : this does not hold in case of negative curvature. Nonetheless, we are able to prove that if the angle $c$ is close enough to $1$, then it is stable, the tool employed here is a Bochner method which can be seen as the counterpart of pointwise estimates due to maximum principles. There is no doubt that this method leads to rigidity as well : \cite{Der-Mod}.

\begin{theo}\label{lic-op-dis-Bochner}
Let $(M,g,\nabla f)$ be an expanding gradient Ricci soliton satisfying $(\Hyp)$ such that
\begin{eqnarray*}
\quad \inf_MR+\frac{n}{2}-2\|\Rm(g)\|_{L^{\infty}}>0.
 \end{eqnarray*}
  Then the bottom of the spectrum of $L$ is positive. In particular, it is dynamically stable.
\end{theo}

\begin{rk}
Note that this condition is satisfied if the curvature operator is bounded by a sufficiently small constant, only depending on the dimension $n$ and this is not a vacuous condition since Ricci expanders are normalized by their very definition. This is in sharp contrast with the shrinking case : Munteanu and Wang \cite{Mun-Wan-Cur} have shown that if the Ricci curvature of a shrinking gradient Ricci soliton is small enough then it is isometric to Euclidean space.
\end{rk}

\begin{rk}
One can ask if there is a finite threshold depending on the angle $c>1$ of the Bryant expanding gradient Ricci solitons concerning their rigidity/stability ?
\end{rk}
\textbf{Acknowledgements.}
The author would like to thank Gilles Carron for his comments on a first draft. This paper benefited from numerous conversations with Peter Topping, Panagiotis  Gianniotis and Felix Schulze.\\

The author is supported by the EPSRC on a Programme Grant entitled ‘Singularities of Geometric Partial Differential Equations’ (reference number EP/K00865X/1).

\section{The linear case}
This section can be skipped for a first reading. This section will not be used for the sequel. The purpose of this section is to give some feeling to the reader with these flows. We begin with a general computation concerning the evolution of the weighted divergence of a solution to the linearized case of (MRHF).

\begin{theo}
Let $(M,g,\nabla f)$ be a gradient Ricci soliton, i.e. 
\begin{eqnarray*}
\Ric+\frac{\epsilon g}{2}=\nabla^2 f,
\end{eqnarray*}
 with $\epsilon\in\{-1,0,1\}$.
Let $h$ be a symmetric $2$-tensor satisfying the heat equation $\partial_t h=\Delta_fh+2\Rm\ast h$.
Then the vector field $\div_f(h)$ satisfies 
\begin{eqnarray*}
\left(\partial_t-\Delta_f+\frac{\epsilon}{2}\right)(\div_f(h))=0.
\end{eqnarray*}
\end{theo}

\begin{proof}
First, note that
\begin{eqnarray*}
\partial_t(h(\nabla f))_i&=&(\partial_th_{ik})\nabla_kf=\Delta_fh_{ik}\nabla_kf+2(\Rm\ast h)_{ik}\nabla_kf\\
&=&\Delta(h_{ik}\nabla_kf)+\nabla_jh_{ik}\nabla_kf\nabla_jf+2(\Rm\ast h)_{ik}\nabla_kf\\
&&-2\nabla_lh_{ik}\nabla_l\nabla_kf-h_{ik}\Delta(\nabla_kf)\\
&=&\Delta(h(\nabla f)_i)+\nabla_j(h(\nabla f)_i)\nabla_jf-h_{ik}\nabla_j\nabla_kf\nabla_jf+2(\Rm\ast h)_{ik}\nabla_kf\\
&&-2\nabla_lh_{ik}\nabla_l\nabla_kf-h_{ik}\Delta(\nabla_kf)\\
&=&\Delta_f(h(\nabla f)_i)-\frac{\epsilon}{2}h(\nabla f)_i-2h(\Ric(\nabla f))_i+2(\Rm\ast h)_{ik}\nabla_kf\\
&&-2\nabla_lh_{ik}\nabla_l\nabla_kf-h_{ik}(\nabla_k\Delta f+\Ric(\nabla f)_k)\\
&=&\Delta_f(h(\nabla f)_i)-\frac{\epsilon}{2}h(\nabla f)_i-2h(\Ric(\nabla f))_i+2(\Rm\ast h)_{ik}\nabla_kf\\
&&-\epsilon\div h-2\nabla_lh_{ik}\Ric_{lk}-h(\nabla R)_i,
\end{eqnarray*}
Hence, by the soliton identities \ref{sol-id-egs},
\begin{eqnarray*}
\partial_t(h(\nabla f))=\Delta_f(h(\nabla f))-\frac{\epsilon}{2}h(\nabla f)+2(\Rm\ast h)(\nabla f)-\epsilon\div h-2\nabla h\ast\Ric,
\end{eqnarray*}
where $(\Rm\ast h)(\nabla f)_i:=\Rm_{ijlk}h_{jl}\nabla_kf$ and $(\nabla h\ast\Ric)_i=\nabla_lh_{ik}\Ric_{lk}$.

Now,
\begin{eqnarray*}
\partial_t\div h =\div(\Delta_fh+2\Rm\ast h).
\end{eqnarray*}
First of all,
\begin{eqnarray*}
\div(\nabla_{\nabla f}h)&=&\nabla_i(\nabla_kf\nabla_kh_{\cdot i})=\nabla_i\nabla_kf\nabla_kh_{\cdot i}+\nabla_kf\nabla_i\nabla_kh_{\cdot i}\\
&=&\frac{\epsilon\div h}{2}+\Ric_{ik}\nabla_kh_{\cdot i}+\nabla_{\nabla f}(\div h)-\nabla_kf(\Rm^p_{ik\cdot}h_{pi} +\Rm^p_{iki}h_{p\cdot})\\
&=&\frac{\epsilon\div h}{2}+\Ric\ast\nabla h+\nabla_{\nabla f}(\div h)+h(\Ric(\nabla f))-(\Rm\ast h)(\nabla f)
\end{eqnarray*}
Also,
\begin{eqnarray*}
\div\Delta h&=&\nabla_i\nabla_k\nabla_kh_{\cdot i}=\Delta(\div h)-\nabla_k(\Rm^p_{ik\cdot}h_{pi} +\Rm^p_{iki}h_{p\cdot})+[\nabla_i,\nabla_k]\nabla_kh_{\cdot i}\\
&=&\Delta(\div h)+\nabla_k(\Ric_{kp}h_{p\cdot}-(\Rm\ast h)_{k\cdot})+[\nabla_i,\nabla_k]\nabla_kh_{\cdot i}\\
&=&\Delta(\div h)+h(\div\Ric)+\Ric\ast\nabla h-\div(\Rm\ast h)+[\nabla_i,\nabla_k]\nabla_kh_{\cdot i}.
\end{eqnarray*}
Note that 
\begin{eqnarray*}
[\nabla_i,\nabla_k]\nabla_kh_{\cdot i}&=&-\Rm^p_{ikk}\nabla_ph_{\cdot i}-\Rm^p_{iki}\nabla_kh_{\cdot p}-\Rm^p_{ik\cdot}\nabla_kh_{pi}\\
&=&-\Ric_{ip}\nabla_ph_{\cdot i}+\Ric_{kp}\nabla_kh_{\cdot p}-\Rm^p_{ik\cdot}\nabla_kh_{pi}\\
&=&-\Rm^p_{ik\cdot}\nabla_kh_{pi},
\end{eqnarray*}
and, by the soliton identities,
\begin{eqnarray*}
\div(\Rm\ast h)&=&\nabla_i(\Rm_{\cdot kli}h_{kl})=\div(\Rm)_{\cdot kl}h_{kl}+\Rm_{\cdot kli}\nabla_ih_{kl}\\
&=&\Rm_{\cdot kjl}\nabla_jfh_{kl}+\Rm_{\cdot kli}\nabla_ih_{kl}\\
&=&-(\Rm\ast h)(\nabla f)+\Rm_{\cdot kli}\nabla_ih_{kl}.
\end{eqnarray*}
Therefore, $[\nabla_i,\nabla_k]\nabla_kh_{\cdot i}=-\Rm^p_{ik\cdot}\nabla_kh_{pi}.$\\

Finally, again, by the soliton identities and the first Bianchi identity,
\begin{eqnarray*}
\partial_t(\div_fh)&=&\Delta_f(\div_fh)-\frac{\epsilon\div_fh}{2}+h(\div \Ric+\Ric(\nabla f))\\
&&+\div(\Rm\ast h)+[\nabla_i,\nabla_k]\nabla_kh_{\cdot i}+\Rm\ast h(\nabla f)\\
&=&\Delta_f(\div_fh)-\frac{\epsilon\div_fh}{2}.
\end{eqnarray*}

\end{proof}

With the help of the Gronwall lemma, one can prove the following :

\begin{theo}
Let $(M,g,\nabla f)$ be a gradient Ricci soliton, i.e. 
\begin{eqnarray*}
\Ric+\frac{\epsilon g}{2}=\nabla^2 f,
\end{eqnarray*}
 with $\epsilon\in\{-1,0,1\}$ which is strictly stable.\\
  
Let $(h_t)_{t\in[0,T]}$, with $T>0$, be a one-parameter family of $2$-symmetric tensor satisfying the heat equation $\partial_t h_t=\Delta_fh_t+2\Rm\ast h_t$ such that, for any $t\in[0,T]$,
\begin{eqnarray*}
h_t\in L^2_{f}\quad\mbox{and}\quad \div_{f}h_0\in L^2_{f}.
\end{eqnarray*}

Then $h_t$ and the vector field $\div_fh_t$ satisfy for any $t\in[0,T]$,
\begin{eqnarray*}
\|\div_fh_t\|^2_{L^2_{f}}&\leq& e^{-\epsilon t}\|\div_fh_0\|^2_{L^2_{f}}\\
\|h_t\|^2_{L^2_{f}}&\leq&\left(\| h_0\|^2_{L^2_{f}}+\|\div_fh_0\|^2_{L^2_{f}}\int_0^te^{2(\lambda-\epsilon)u}du\right)e^{-2\lambda t}
\end{eqnarray*}
\end{theo}

\begin{rk}
It seems that, in the case of the (MRHF) in the expanding case, the $L^2_{f}$-decay of $\div_{f_0}h_t$ is independent of $\lambda$. Nonetheless, the author was not able to prove it in general.
\end{rk}

\section{Existence}
Fix now once for all a background gradient Ricci soliton $(M,g_0,\nabla^0f_0)$ satisfying $(\Hyp)$.\\
 
We begin by investigating the $C^0$-evolution of the difference $g(t)-g_0$ : 

\begin{lemma}\label{pt-est}
Let $g(\cdot)$ be a solution to the Ricci harmonic map heat flow on $(0,T)$ which is $\epsilon$-close to $g_0$. Then,
\begin{eqnarray*}
\frac{\partial}{\partial t}\arrowvert h\arrowvert^2&\leq& g^{-1}\ast\nabla^0\nabla^0\arrowvert h\arrowvert^2+X^0\cdot\arrowvert h\arrowvert^2-2(1-\epsilon)\arrowvert \nabla^0h\arrowvert^2\\
&&+4\Rm(g_0)(h,h)+4\epsilon\arrowvert\Rm(g_0)(h,h)\arrowvert,
\end{eqnarray*}
where 
\begin{eqnarray*}
&&h:=g-g_0,\\ 
&&g^{-1}\ast\nabla^0\nabla^0\arrowvert h\arrowvert^2:=g^{ij}\nabla^0_i\nabla^0_j\arrowvert h\arrowvert^2,\\
&&\Rm(g_0)(h,h):=\Rm(g_0)_{iklj}h_{kl}h_{ij}.
\end{eqnarray*}
\end{lemma}

\begin{proof}
We do the proof in the case of an expanding gradient Ricci soliton ($\epsilon=1$).
According to Shi \cite{Shi-Def},
\begin{eqnarray*}
\frac{\partial}{\partial t}g_{ij}&=&g^{ab}\nabla^0_a\nabla^0_bg_{ij}-g^{kl}g_{ip}\Rm(g_0)_{jklp}-g^{kl}g_{jp}\Rm(g_0)_{iklp}\\
&&+(\nabla^0g\ast\nabla^0g)_{ij}+\Li_{X^0}(g)_{ij}-g_{ij},
\end{eqnarray*}
where, if $A$ and $B$ are two tensors, $A\ast B$ means some linear combination of contractions of $A\otimes B$.
Therefore,
\begin{eqnarray*}
\frac{\partial}{\partial t}\arrowvert h\arrowvert^2&=&2\left\langle\frac{\partial}{\partial t}g,h\right\rangle\\
&=&2\langle g^{-1}\ast\nabla^0\nabla^0g-\Rm(g_0)\ast g^{-1}\ast g,h\rangle\\
&&-2\langle g-\Li_{X^0}(g),h\rangle+\langle \nabla^0g\ast\nabla^0g,h\rangle\\
&\leq&g^{-1}\ast\nabla^0\nabla^0\arrowvert h\arrowvert^2-2(1-\epsilon)\arrowvert\nabla^0 h\arrowvert^2\\
&&-2\langle g-\Li_{X^0}(g),h\rangle-2\langle\Rm(g_0)\ast g^{-1}\ast g,h\rangle,
\end{eqnarray*}
where $(\Rm(g_0)\ast g^{-1}\ast g)_{ij}:=g^{kl}g_{ip}\Rm(g_0)_{jklp}+g^{kl}g_{jp}\Rm(g_0)_{iklp}$. Now, using the soliton equation satisfied by $g_0$,
\begin{eqnarray*}
\langle g-\Li_{X^0}(g),h\rangle=\arrowvert h\arrowvert^2-\langle\Li_{X^0}(h),h\rangle-2\langle\Ric(g_0),h\rangle.
\end{eqnarray*}

Moreover, as 
\begin{eqnarray*}
\Li_{X^0}(h)&=&\nabla^0_{X^0}h+\frac{1}{2}(\Li_{X^0}(g_0)\otimes h+h\otimes\Li_{X^0}(g_0)),\\
&=&\nabla^0_{X^0}h+h+\Ric(g_0)\otimes h+h\otimes \Ric(g_0),
\end{eqnarray*}
we get 
\begin{eqnarray*}
2\langle g-\Li_{X^0}(g),h\rangle=-X^0\cdot\arrowvert h\arrowvert^2-4\langle\Ric(g_0)\otimes h,h\rangle-4\langle\Ric(g_0),h\rangle.
\end{eqnarray*}
Hence,
\begin{eqnarray*}
\frac{\partial}{\partial t}\arrowvert h\arrowvert^2&\leq& g^{-1}\ast\nabla^0\nabla^0\arrowvert h\arrowvert^2-2(1-\epsilon)\arrowvert\nabla^0 h\arrowvert^2-2\langle\Rm(g_0)\ast g^{-1}\ast g,h\rangle\\
&&+4\langle\Ric(g_0),h\rangle+X^0\cdot\arrowvert h\arrowvert^2+2\langle\Li_{X^0}(g_0)\otimes h,h\rangle\\
&\leq&g^{-1}\ast\nabla^0\nabla^0\arrowvert h\arrowvert^2+X^0\cdot\arrowvert h\arrowvert^2-2(1-\epsilon)\arrowvert\nabla^0 h\arrowvert^2\\
&&+4\langle\Ric(g_0)\otimes h,h\rangle+4\langle\Ric(g_0),h\rangle-2\langle\Rm(g_0)\ast g^{-1}\ast g,h\rangle.\\
\end{eqnarray*}
Finally, let us compute the curvature term in coordinates at a fixed point and a fixed time such that $({g_0}_{ij})=(\delta_{ij})$ and $(g_{ij})=(\lambda_i\delta_{ij})$ where $\lambda_i$ are  positive numbers.
\begin{eqnarray*}
\langle\Rm(g_0)\ast g^{-1}\ast g,h\rangle&=&2\Rm(g_0)_{ikki}\lambda_k^{-1}\lambda_i(\lambda_i-1)\\
&=&2\Rm(g_0)_{ikki}(\lambda_k^{-1}-1)\lambda_i(\lambda_i-1)\\
&&+2\Rm(g_0)_{ikki}\lambda_i(\lambda_i-1)\\
&=&-2\Rm(g_0)(\lambda_k-1)(\lambda_i-1)\frac{\lambda_i}{\lambda_k}\\
&&+2\Rm(g_0)_{ikki}(\lambda_i-1)^2+2\Rm(g_0)_{ikki}(\lambda_i-1)\\
&=&-2\Rm(g_0)(h,h)+\Rm(g_0)\ast h^{\ast 3}\\
&&+2\langle\Ric(g_0)\otimes h,h\rangle+2\langle\Ric(g_0),h\rangle.
\end{eqnarray*}

Summing it up, we get
\begin{eqnarray*}
\frac{\partial}{\partial t}\arrowvert h\arrowvert^2&\leq& g^{-1}\ast\nabla^0\nabla^0\arrowvert h\arrowvert^2+X^0\cdot\arrowvert h\arrowvert^2\\
&&-2(1-\epsilon)\arrowvert\nabla^0 h\arrowvert^2+4\Rm(g_0)(h,h)+4\epsilon\arrowvert\Rm(g_0)(h,h)\arrowvert.
\end{eqnarray*}
\end{proof}

We use this previous differential inequality to solve Dirichlet problems of (MRHF) on geodesic balls of $(M,g_0)$.
\begin{coro}\label{a-priori-est-dir}
Let $(g(t))_{t\in[0,T)}$ with $0<T<\infty$, be a solution of (MRHF) on $B_{g_0}(p,R)$ with $$g(t)\mid_{\partial B_{g_0}(p,R)}=g_0\mid_{\partial B_{g_0}(p,R)}.$$ Let $\delta$ be a positive number. Then there exists $\epsilon=\epsilon(n,T,\delta,k_0))>0$ such that $$\|g(0)-g_0\|_{L^{\infty}(B_{g_0}(p,R))}\leq \epsilon$$ implies
\begin{eqnarray*}
\|g(t)-g_0\|_{L^{\infty}(B_{g_0}(p,R)\times [0,T))}\leq \delta. 
\end{eqnarray*}
\end{coro}

\begin{proof}
This is exactly lemma $2.3$ of \cite{Sch-Sch-Sim}.
\end{proof}

Now, if we manage to prove a priori that our solution is sufficiently close to the background metric $g_0$, this solution will exist for all time. This is the content of the following theorem which is word for word (its content and the proof) theorem $2.4$ of \cite{Sch-Sch-Sim}.
\begin{defn}
A solution to (MRHF) $(g(t))_{t\in[0,T)}$ is \textbf{$\delta$-maximal} if either $T$ is infinite and $(g(t))_{t\in[0,+\infty)}$ is $\delta$-close to $g_0$ either we can extend the solution on $M\times [0,T+\tau)$ with $\tau=\tau(n,\sup_M\arrowvert\Rm(g)\arrowvert)$ positive and $\|g(T)-g_0\|_{L^{\infty}(M)}=\delta.$
\end{defn}

\begin{theo}\label{delta-max-sol}
There exists a positive constant $\tilde{\delta}=\tilde{\delta}(n,k_0)$ such that the following holds.
Let $0<\epsilon<\delta\leq\tilde{\delta}$. Then every metric $g(0)$ $\epsilon$-close to $g_0$ has a $\delta$-maximal solution $g(t)_{t\in[0,T(g_0))}$ with $T(g_0)$ positive which is $\delta$-close for all $t\in[0,T(g_0))$. 
\end{theo}

With the help of the a priori estimates of corollary \ref{a-priori-est-dir}, a useful corollary for the sequel consists in the following :
\begin{coro}\label{amorce-tps-inf}
With the same assumptions as in theorem \ref{delta-max-sol}, if $T>0$ is given, and if $\epsilon(n,T,\delta,k_0)$ is small enough, then the solution $(g(t))_{t\in[0,T_{g_0}+\tau)}$ satisfies $T_{g_0}\geq T$.
\end{coro}

\section{Convergence}
Again, consider a background gradient Ricci soliton $(M,g_0,\nabla^0f_0)$ satisfying $(\Hyp)$.\\
If $(g(t))_{t\in[0,T)}$ is close enough to $g_0$, we consider the truncated $L_{f_0}^2$ norm of the difference $h(t):=g(t)-g_0$ as a possible Lyapunov function as in \cite{Sch-Sch-Sim}.\\

\begin{prop}\label{Lya-fct}
Assume $(M,g_0,\nabla^0f_0)$ is strictly stable. Let $(g(t))_{t\in[0,T)}$ be a solution to (MRHF) on $B_{g_0}(p,R)$ with $g(t)=g_0$ on $S_{g_0}(p,R)\times[0,T)$. Assume that $\|g-g_0\|_{L^{\infty}(B_{g_0}(p,R)\times[0,T))}\leq \epsilon\leq \epsilon(\lambda,k_0)$. Then we have
\begin{eqnarray*}
\int_{B_{g_0}(p,R)}\arrowvert g(t)-g_0\arrowvert^2d\mu_{f_0}\leq e^{-\tilde{\lambda}t}\int_{B_{g_0}(p,R)}\arrowvert g(0)-g_0\arrowvert^2d\mu_{f_0},
\end{eqnarray*}
where $\tilde{\lambda}=\tilde{\lambda}(\epsilon)<2\lambda$ such that $\lim_{\epsilon\rightarrow +\infty}\tilde{\lambda}(\epsilon)=2\lambda$. 
\end{prop}

\begin{proof}
These estimates follow by integrating the pointwise estimates of lemma \ref{pt-est} and by using integration by parts :
\begin{eqnarray*}
&&\frac{d}{dt}\int_{B_{g_0}(p,R)}\arrowvert h\arrowvert^2d\mu_{f_0}\leq\int_{B_{g_0}(p,R)}g^{-1}\ast\nabla^0\nabla^0\arrowvert h\arrowvert^2+X^0\cdot\arrowvert h\arrowvert^2d\mu_{f_0}\\
&&+\int_{B_{g_0}(p,R)}-2(1-\epsilon)\arrowvert \nabla^0h\arrowvert^2+4\Rm(g_0)(h,h)+4\epsilon\arrowvert\Rm(g_0)(h,h)\arrowvert d\mu_{f_0}\\
&&\leq\int_{B_{g_0}(p,R)}\Delta_{f_0}\arrowvert h\arrowvert^2+\left(g^{-1}-g_0^{-1}\right)\ast\nabla^0\nabla^0\arrowvert h\arrowvert^2d\mu_{f_0}\\
&&\int_{B_{g_0}(p,R)}-2(1-\epsilon)\arrowvert \nabla^0h\arrowvert^2+4\Rm(g_0)(h,h)+4\epsilon\arrowvert\Rm(g_0)(h,h)\arrowvert d\mu_{f_0}\\
&&\leq-2(1-\epsilon)\int_{B_{g_0}(p,R)}\arrowvert \nabla^0h\arrowvert^2-2\Rm(g_0)(h,h)d\mu_{f_0}\\
&& +c_0\epsilon\int_{B_{g_0}(p,R)}\arrowvert h\arrowvert^2d\mu_{f_0}-\int_{B_{g_0}(p,R)}\langle\div_{f_0}(g^{-1}),\nabla^0\arrowvert h\arrowvert^2\rangle d\mu_{f_0},
\end{eqnarray*}
where $c_0=c(k_0).$
To get rid of the divergence term, as the vector field $\nabla^0 f_0$ is not bounded, we cannot simply bound this term by the norm of the covariant derivative of $h$. Therefore, we need the following identity inspired by the work of Koiso \cite{koi-ein} :

\begin{eqnarray*}
2<\nabla^0 T,\nabla^0 T>_{L^2_{f_0}}&=&<\Cod(T),\Cod(T)>_{L^2_{f_0}}+2<\div_{f_0}(T),\div_{f_0}(T)>_{L^2_{f_0}}\\
&&-2<\Ric_{f_0}(T),T>_{L^2_{f_0}}+2<\Rm(g_0)\ast T,T>_{L^2_{f_0}}\\
&=&\|\Cod(T)\|^2_{L_{f_0}^2}+2\|\div_{f_0}(T)\|^2_{L_{f_0}^2}\\
&&+\|T\|^2_{L^2_{f_0}}+2<\Rm(g_0)\ast T,T>_{L^2_{f_0}},
\end{eqnarray*}
for any $T\in C_0^{\infty}(M,S^2T^*M)$ where $\Cod(T)_{ijk}:=\nabla^0_iT_{jk}-\nabla^0_jT_{ik}$ is the Codazzi tensor associated to $T$. Hence the $L^2_{f_0}$-norm of the divergence of $T$ is bounded by the $L^2_{f_0}$-norm of $T$ and $\nabla^0 T$. Therefore, by applying the Young inequality,

\begin{eqnarray*}
\frac{d}{dt}\int_{B_{g_0}(p,R)}\arrowvert h\arrowvert^2d\mu_{f_0}&\leq&-(2-c_0\epsilon)\int_{B_{g_0}(p,R)}\arrowvert \nabla^0h\arrowvert^2-2\Rm(g_0)(h,h)d\mu_{f_0}\\
&&+c_0\epsilon\int_{B_{g_0}(p,R)}\arrowvert h\arrowvert^2d\mu_{f_0}.
\end{eqnarray*}

Now, as $(M,g_0,\nabla^0f_0)$ is strictly stable, and if $\epsilon$ is small enough, we get : 

\begin{eqnarray*}
&&\frac{d}{dt}\int_{B_{g_0}(p,R)}\arrowvert h\arrowvert^2d\mu_{f_0}\leq-\left((2-c_0\epsilon)\lambda-c_0\epsilon\right)\int_{B_{g_0}(p,R)}\arrowvert h\arrowvert^2d\mu_{f_0},
\end{eqnarray*}
with $(2-c_0\epsilon)\lambda-c_0\epsilon>0$.

\end{proof}

Since we are building a solution with the help of Dirichlet exhaustions, the previous estimate extends to $M\times[0,T)$ : 

\begin{coro}
Assume $(M,g_0,\nabla^0f_0)$ is strictly stable. Let $T>0$. Assume $g(0)\in L_{f_0}^2(M,S^2T^*M)$. 

Then there exists $\epsilon_0=\epsilon_0(n,T,k_0,\lambda)$ such that if $g(0)$ is $\epsilon_0$-close to $g_0$ then a solution to (MRHF) $(g(t))_{t\in[0,T)}$ on $M$ exists, is $\epsilon$-close with $\epsilon$ as in proposition \ref{Lya-fct} and satisfies
\begin{eqnarray*}
\|g(t)-g_0\|^2_{L_{f_0}^2(M,S^2T^*M)}\leq e^{-\tilde{\lambda}t}\|g(0)-g_0\|^2_{L_{f_0}^2(M,S^2T^*M)},
\end{eqnarray*}
for all $t\in[0,T)$ with $\tilde{\lambda}$ defined in proposition \ref{Lya-fct}.
\end{coro}

Such an exponential decay of the $L_{f_0}^2$-norm of $g(t)-g_0$ implies an exponential decay of the supremum norm of $g(t)-g_0$. Before stating this result, we need to uniformly control the covariant derivatives of $g(t)$ in time :
\begin{lemma}\label{int-est}
Let $(g(t))_{t\in[0,T)}$ be a solution of (MRHF) on $M$ such that $g(t)$ is $\epsilon$-close to $g_0$ for any $t\in[0,T)$ with $\epsilon\leq\epsilon(n,k_0).$ Then,
\begin{eqnarray*}
\sup_M\arrowvert\nabla^{0,j}g(t)\arrowvert\leq\frac{c(n,j,k_0)}{t^j},
\end{eqnarray*}
for all $t\in(0,\min\{1,T\}].$
\end{lemma}

\begin{proof}
The proof is almost the same as the proof of theorem $4.3$ in \cite{Sim-C0-Def}. The only thing that could cause trouble is the presence of the vector field $\nabla^0 f_0$ that is unbounded in the expanding case. But it turns out that it does not by considering adequate cut-off functions built with the help of the potential function. To convince the reader, we reproduce the proof for the first covariant derivative in the expanding case.
 \begin{itemize}
\item First of all, by using commutation formula and identities on gradient Ricci solitons, one has on $M\times [0,T]$,
\begin{eqnarray*}
\partial_t\arrowvert\nabla^0g\arrowvert_{g_0}^2&\leq& g^{-1}\ast\nabla^{0,2}\arrowvert\nabla^0g\arrowvert_{g_0}^2+<\nabla^0 f_0,\nabla^0\arrowvert\nabla^0g\arrowvert_{g_0}^2>\\
&&+c_0\arrowvert\nabla^0g\arrowvert_{g_0}^2+c_0+c(n)\arrowvert\nabla^0g\arrowvert_{g_0}^4,
\end{eqnarray*}
where $c_0=c(n,k_0)$.\\
\item Secondly, define $\phi:=a+\sum_{k=1}^n\lambda_k^m(t)$ where $\lambda_k(t)$ are the eigenvalues of $g(t)$ and $m$ is an integer. Adapting the computations of Shi \cite{Shi-Def}, one gets on $M\times [0,T]$,
\begin{eqnarray*}
\partial_t\phi\leq g^{-1}\ast\nabla^{0,2}\phi+<\nabla^0f_0,\nabla^0\phi>-\frac{\phi^2}{2}+c_0,
\end{eqnarray*}
with suitable $a$ and $m$ depending on $n$ and $k_0$.\\
\item Hence, if $F(t,x):=t\phi(t,x)\arrowvert\nabla^0g\arrowvert_{g_0}^2(t,x)$, we get
\begin{eqnarray*}
\partial_tF\leq g^{-1}\ast\nabla^{0,2}F+<\nabla^0f_0,\nabla^0F>-\frac{F^2}{2t}+c_0t+\frac{F}{t},
\end{eqnarray*}
on $M\times (0,T]$.\\
\item Let $\eta:M\rightarrow[0,1]$ be a smooth compactly supported function on $M$ defined by : $\eta(x):=\tilde{\eta}(\sqrt{f_0(x)}/r)$ with $r>0$ and $\tilde{\eta}:[0,+\infty[\rightarrow[0,1]$ is a smooth compactly supported function such that
\begin{eqnarray*}
\tilde{\eta}\arrowvert_{[0,1/2]}\equiv 1,\quad \tilde{\eta}\arrowvert_{[1,+\infty[}\equiv 0,\quad \tilde{\eta}'\leq 0,\quad \frac{\tilde{\eta}'^2}{\tilde{\eta}}\leq c,\quad \tilde{\eta}''\geq -c,
\end{eqnarray*}
where $c$ is a universal constant.
Now,
\begin{eqnarray*}
\nabla^0 \eta=\frac {\tilde{\eta}'}{r}\nabla^0 \sqrt{f_0},\quad\nabla^{2,0}\eta=\frac{\tilde{\eta}''}{r^2}\nabla^0 \sqrt{f_0}\otimes\nabla^0 \sqrt{f_0}+\frac{\tilde{\eta}'}{r}\nabla^{0,2} \sqrt{f_0}.
\end{eqnarray*}
Therefore, by considering a point $(t_0,x_0)$ where the function $\eta F$ attains its maximum on $M\times[0,T]$, using the relations 
\begin{eqnarray*}
\nabla^0(\eta F)(t_0,x_0)=0, \quad \partial_t(\eta F)(t_0,x_0)\geq 0,\quad g^{-1}\ast\nabla^{0,2}(\eta F)(t_0,x_0)\leq 0,
\end{eqnarray*}
with the help of the previous differential inequality, one gets a polynomial of degree $2$ in $(\eta F)(t_0,x_0)$ with coefficients depending on $n$, $k_0$, $t$, $1/r$ and $<\nabla^0f_0,\nabla^0\eta>$ which is bounded since $f_0$ grows at most quadratically on an expanding gradient Ricci soliton. Hence the result if we restrict time to be lower than $\min\{1,T\}$.
\end{itemize}
 The higher covariant derivative estimates can be proved in a similar way : see lemma $4.2$ and theorem $4.3$ of \cite{Sim-C0-Def}.
\end{proof}

Now, we have
\begin{theo}\label{sup-est}
Assume $(M,g_0,\nabla^0f_0)$ is strictly stable. 
Assume $(g(t))_{t\in[0,T)}$ is a solution to (MRHF) with 
\begin{eqnarray*}
\|g(0)-g_0\|_{L^2_{f_0}(M,S^2T^*M)}=:I<+\infty,\quad \|g(0)-g_0\|_{L^{\infty}(M,S^2T^*M)}\leq \epsilon,
\end{eqnarray*}
 with $\epsilon$ as in proposition \ref{Lya-fct}. Then,
\begin{eqnarray*}
\|g(t)-g_0\|_{L^{\infty}(M,S^2T^*M)}\leq C(n,I,\lambda,k_0) \exp\left(-\frac{\tilde{\lambda}t}{n+2}\right),
\end{eqnarray*}
for $t\in[0,T)$.
\end{theo}

\begin{proof}
Such estimates clearly holds for some interval $[0,\tau]$ where $\tau=\tau(n,\epsilon)$. By the interior estimates given by lemma \ref{int-est}, there exists a constant $c=c(n,\lambda,k_0)$ such that
\begin{eqnarray}\label{bd-cov-der}
\arrowvert\nabla^0g(t)\arrowvert_{g_0}\leq c,
\end{eqnarray}
for $t\in[\tau,T)$.
Fix $t\in[\tau,+\infty)$ and consider $m(t):=\|g(t)-g_0\|_{L^{\infty}(M,S^2T^*M)}$. Choose a point $x(t)\in M$ such that 
\begin{eqnarray*}
\frac{1}{2}\|g(t)-g_0\|_{L^{\infty}(M,S^2T^*M)}\leq\|g(t)-g_0\|_{g_0}(x(t)).
\end{eqnarray*}
Because of the bound (\ref{bd-cov-der}), we have 
\begin{eqnarray*}
\frac{1}{4}\|g(t)-g_0\|_{L^{\infty}(M,S^2T^*M)}\leq\|g(t)-g_0\|_{g_0}(y),
\end{eqnarray*}
for any $y\in B_{g_0}(x(t),m(t)/(4c))$. Therefore,
\begin{eqnarray*}
\vol_{f_0}B_{g_0}(x(t),m(t)/(4c))m(t)^2\leq 4Ie^{-\tilde{\lambda}t}.
\end{eqnarray*}
In the expanding case, $f$ grows quadratically in the distance to a fixed point $p$ in $M$ by $(\Hyp)$.
If $m(t)/4c\geq 1$, then, by the Bishop-Gromov theorem, if $\Ric_{g_0}\geq -(n-1)K_0$ on $M$,
\begin{eqnarray*}
&&\vol_{f_0}B_{g_0}(x(t),m(t)/(4c))\geq \vol_{f_0}B_{g_0}(x(t),1)\\
&&\geq e^{f_0(x(t))-\sqrt{f_0(x(t))}-c_0}\vol_{g_0}B_{g_0}(x(t),1)\\
&&\geq e^{f_0(x(t))-\sqrt{f_0(x(t))}-c_0}\frac{\vol_{-K_0}B_{-K_0}(0,1)}{\vol_{-K_0}B_{-K_0}(0,r_p(x(t))+1)}\vol_{g_0}B_{g_0}(p,1)\\
&&\geq c_0>0,
\end{eqnarray*}
where $c_0=c(n,g_0)$, since the volume $\ln\vol_{-K_0}B_{-K_0}(0,r_p(x(t))+1)$ grows at most linearly in the distance to a fixed point $p$ in $M$.

If $m(t)/4c\leq 1$, then a similar argument shows that,
\begin{eqnarray*}
\vol_{f_0}B_{g_0}(x(t),m(t)/(4c))\geq c_0(m(t))^n.
\end{eqnarray*}
Finally, in any cases,
\begin{eqnarray*}
m(t)\leq c_0 I \exp\left(-\frac{\tilde{\lambda}t}{n+2}\right),
\end{eqnarray*}
for any $t\in[\tau,T)$.

A similar argument works in the steady case with non negative Ricci curvature ($K_0=0$) with the help of lemma \ref{gro-pot-fct}.

\end{proof}

We are now in a position to prove long time existence and exponential convergence to $g_0$ :

\begin{theo}\label{conv-MRHF}
Assume $(M,g_0,\nabla^0f_0)$ is strictly stable.
For any $I>0$, there exists $\epsilon(n,I,k_0,\lambda)>0$ such that the following holds. If $g(0)$ satisfies  
\begin{eqnarray*}
\|g(0)-g_0\|_{L^2_{f_0}(M,S^2T^*M)}\leq I\quad \mbox{and}\quad \|g(0)-g_0\|_{L^{\infty}(M,S^2T^*M)}\leq \epsilon,
\end{eqnarray*}
 then there exists a solution to (MRHF) such that
\begin{eqnarray*}
\|g(t)-g_0\|_{L^{\infty}(M,S^2T^*M)}\leq C(n,I,\lambda,k_0) \exp\left(-\frac{\tilde{\lambda}t}{2(n+2)}\right),
\end{eqnarray*}
for any time $t$. Moreover, for any $k\in\mathbb{N}$,
\begin{eqnarray*}
\|g(t)-g_0\|_{C^k(M,S^2T^*M)}\leq C(n,k,I,\lambda,k_0)e^{-\lambda_kt},
\end{eqnarray*}
where $\lambda_k\in\left(0,\tilde{\lambda}/(n+2)\right)$ is arbitrary. 
\end{theo}

\begin{proof}
Again, this is exactly the proof of theorem $3.4$ of \cite{Sch-Sch-Sim}. In view of theorem \ref{delta-max-sol}, the only thing to check is that this solution $(g(t))_{t\in[0,T)}$ to (MRHF) we built is $\tilde{\delta}$-close to $g_0$ for any time with the notations of theorem \ref{delta-max-sol}. If $T>0$ is given, we can find some $\epsilon(n,T,k_0)>0$ such that  $(g(t))_{t\in[0,T]}$ is $\tilde{\delta}$-close to $g_0$. Now, if $T$ is large enough, $g(t)$ remains $\tilde{\delta}$-close to $g_0$ for $t\geq T$ by theorem \ref{sup-est}. Hence the existence of the solution for any time. The $L^{\infty}$-estimate follows directly by theorem \ref{sup-est}.

By standard $L^{\infty}$-interpolation inequalities and interior estimates of lemma \ref{int-est}, one can prove the convergence of this solution to $g_0$ in the smooth sense. More precisely, one has to iterate the following inequality :
\begin{eqnarray*}
\|\nabla^{0}T\|^2_{L^{\infty}(M)}\leq c_0 \|T\|_{L^{\infty}(M)}(\|\nabla^{0,2}T\|_{L^{\infty}(M)}+\|\nabla^{0}T\|_{L^{\infty}(M)})
\end{eqnarray*}
for any tensor $T$.

\end{proof}

\section{Back to the Modified Ricci flow}
We are able to go back to the (MRF) by pulling back the solution obtained in theorem \ref{conv-MRHF} by an appropriate one-parameter family of diffeomorphisms of $M$.

\begin{proof}[Proof of theorem \ref{theo-MRF}]
Again, we follow the steps of the proof of theorem $4.1$ of \cite{Sch-Sch-Sim}.
\begin{itemize}
\item First, define $(\psi_t)_t$ to be the flow of the time-dependent vector field $-V(g(t))$ defined in (\ref{de-turck-vect}). By its very definition, $(\psi_t^*g(t))_t$ solves the (MRF) with initial condition $g(0)$. This flow is defined on $[0,+\infty)$ since the vector field $-V(g(t))$ is decaying exponentially in time.
\item Again, as $-V(g(t))$ is decaying exponentially in time., we have for $0\leq t\leq s$, and $x\in M$,
\begin{eqnarray*}
d_{g_0}(\psi_t(x),\psi_{s}(x))\leq\int_t^s\arrowvert V(g(u))\arrowvert_{g_0}(x)du\leq c\int_t^se^{-\lambda_1 u}du.
\end{eqnarray*}
where $c=c(n,\lambda,k_0)$ is independent of $x\in M$. Therefore, there exists a continuous map $\psi_{\infty}:M\rightarrow M$ such that
\begin{eqnarray*}
d_{g_0}(\psi_t(x),\psi_{\infty}(x))\leq c\int_{t}^{+\infty}e^{-\lambda_1 u}du.
\end{eqnarray*}
To prove higher regularity and global $C^k$ convergence, we use the fact that the higher covariant derivatives of the vector field $V$ are exponentially decaying in time uniformly in space.
\item Proving that $\psi_{\infty}$ is actually a diffeomorphism is a bit more involved : 
\begin{itemize}
\item $\psi_{\infty}$ is a local diffeomorphism. Indeed, if $\tilde{g}(t):=\psi_t^*g(t)$ then $(\tilde{g}(t))_t$ solves the (MRF) and $\partial_t\tilde{g}(t)$ decay exponentially  in time uniformly in space. Therefore, $\tilde{g}(t)$ smoothly converges to a smooth metric $g_{\infty}$ on $M$. In particular, on a fixed geodesic ball for $g_0$, and $t$ large enough,
\begin{eqnarray*}
\frac{1}{2}g_{\infty}\leq \psi_t^*g(t)\leq 2\psi_t^*g_0.
\end{eqnarray*}
By letting $t$ go to $+\infty$, we get the result.
\item $\psi_{\infty}$ is bijective. Indeed, it suffices to apply the above arguments to the inverse flow $(\psi_t^{-1})_t$ to prove the existence of a smooth map $\tilde{\psi}_{\infty}:M\rightarrow M$ such that for any $x$ in $M$, $\tilde{\psi}_{\infty}\circ\psi_{\infty}(x)=x=\psi_{\infty}\circ\tilde{\psi}_{\infty}(x)$.
\end{itemize}   
\end{itemize}
\end{proof}

\section{Spectrum of Laplace operators}
\subsection{Preliminaries}
Let $(M,g,\nabla f)$ be a steady or an expanding gradient Ricci soliton satisfying $(\Hyp)$. Let $E$ be a geometric tensor bundle over $M$ in the sense of \cite{Lee-Hyp} : we will only consider the cases where $E$ is either the trivial line bundle over $M$ either the space of symmetric two tensors. We will denote by $\,^{f}H_{k}^{p}(M,E)$ the Banach space of all locally integrable sections $h$ of $E$ such that $\nabla^ih$ is in $L^2(M,E\otimes T^iM)$ for $0\leq i\leq k$, with the norm
\begin{eqnarray*}
 \,^{f}\| h\|_{k,p}:=\left(\sum_{i=0}^k\int_M\arrowvert\nabla^ih\arrowvert^2d\mu_f \right)^{1/2}.
\end{eqnarray*}
We also denote $\,^{f}{\overset{\circ}{H_{k}^p}}(M,E)$ the completion of the set of smooth compactly supported sections of $E$, $C_{0}^{\infty}(M,E)$, in $\,^{f}H_{k}^{p}(M,E)$. Finally, let's introduce the following a priori subspaces of $\,^{f}H_{2}^{2}(M,E)$ as in \cite{Gri-Boo} : $\,^{f}W^{2}(M,E)$ denotes the Hilbert space of all locally integrable sections $h\in\,^{f}H_{1}^{2}(M,E)$ of $E$ such that $\Delta_fh$ is in $L_f^2(M,E)=\,^{f}H_{0}^{2}(M,E)$, with the norm
\begin{eqnarray*}
 \,^{f}\| h\|_{W^2}:= \,^{f}\|h\|_{1,2}+\,^{f}\|\Delta_f h\|_{0,2},
\end{eqnarray*}
and we define $$\,^{f}{\overset{\circ}{W^2}}(M,E)=\left\{h\in\,^{f}{\overset{\circ}{H_{1}^2}}(M,E)\quad|\quad \Delta_fh\in L_f^2(M,E)\right\}$$ equipped with the induced norm denoted by $\,^{f}\| \cdot\|_{\overset{\circ}{W^2}}.$

Firstly, we recall a fundamental result of Grigor'yan [Theorem $4.6$ , \cite{Gri-Boo}] on extension of weighted operators on weighted manifolds (we quote it in the case of gradient Ricci solitons) :
\begin{theo}(Grigor'yan)
Let $(M,g,\nabla f)$ be a gradient Ricci soliton. Then the operator $-\Delta_f|_{C_0^{\infty}(M,E)}$ admits a unique self-adjoint extension to $\,^{f}{\overset{\circ}{W^2}}(M,E)$ whose domain is contained in $\,^{f}{\overset{\circ}{W^2}}(M,E)$. 
\end{theo}
 
 Actually, we are able to precise what is $\,^{f}{\overset{\circ}{W^2}}(M,E)$  :
 \begin{prop}
 Let $(M,g,\nabla f)$ be a steady or an expanding gradient Ricci soliton satisfying $(\Hyp)$. Then
 \begin{enumerate}
 \item $\,^{f}H_{k}^{p}(M,E)=\,^{f}{\overset{\circ}{H_{k}^p}}(M,E)$, for any nonnegative finite integers $k,p$,
 \item $\,^{f}W^2(M,E)=\,^{f}{\overset{\circ}{W^2}}(M,E),$
 \item$\,^{f}H_{2}^{2}(M,E)=\,^{f}{\overset{\circ}{W^2}}(M,E)$.
 \end{enumerate}
 \end{prop}

\begin{proof}
Let $\psi:\mathbb{R}\rightarrow[0,1]$ be a smooth function such that $\psi(t)=1$ if $t\leq 1$ and $\psi(t)=0$ if $t\geq 2$. Define $$\phi_k(x):=\psi\left(\frac{f(x)}{k}\right),$$
for $k>0$ and $x\in M$. Because of $(\Hyp)$, $\phi_k$ is smooth with compact support in $\{f\leq 2k\}$. Moreover, by the soliton identities \ref{sol-id-egs} and \ref{sol-id-sgs}, one can check that $\nabla^i\phi_k$, for $i>0$, is bounded on $M$ with compact support in $\{k\leq f\leq 2k\}$. 
\begin{enumerate}
\item Let $h\in \,^{f}H_{k}^{p}(M,E)$. Then $(\phi_k\cdot h)_k$ converges in $ \,^{f}H_{k}^{p}(M,E)$ by the previous remarks. For any $k\in\mathbb{N}^*$, $\phi_k\cdot h$ has compact support. Therefore, by a convolution argument, one can approximate, for each $k$, $\phi_k\cdot h$ by a smooth compactly supported tensor $h_k$ in the usual unweighted Sobolev space $H_{k}^{p}(M,E)$ hence in the weighted Sobolev space $\,^{f}H_{k}^{p}(M,E)$ since the weighted and unweighted norms  are uniform on fixed compact sets of $M$.
\item The same previous argument applies by using the following remark : $$\Delta_f\phi_k=\Delta_g\phi_k+\psi'\frac{\arrowvert\nabla f\arrowvert^2}{k},$$ is bounded by the soliton identities \ref{sol-id-egs} and \ref{sol-id-sgs} and has compact support in $\{k\leq f\leq 2k\}$.
\item Finally, according to what we proved, it suffices to prove $\,^{f}\overset{\circ}{H_{2}^{2}}(M,E)=\,^{f}{\overset{\circ}{W^2}}(M,E)$. This is achieved with the help of Bochner formulas : 
\begin{eqnarray*}
\Delta_f\arrowvert\nabla h\arrowvert^2=2\arrowvert\nabla^2h\arrowvert^2+2<\nabla\Delta_fh,\nabla h>+<\nabla\Rm(g)\ast h+\nabla h\ast \Rm(g),\nabla h>,
\end{eqnarray*}
for any $h\in C^{\infty}(M,E)$. If $h\in C_0^{\infty}(M,E)$, then an integration by part gives :
\begin{eqnarray*}
\,^{f}\|\Delta_f h\|_{0,2}=\,^{f}\|\nabla^2 h\|_{0,2}+\,^{f}<\nabla\Rm(g)\ast h+\nabla h\ast \Rm(g),\nabla h>_{0,2},
\end{eqnarray*}
which gives the expected result since $\Rm(g)$ is bounded by assumption $(\Hyp)$ and so is $\nabla\Rm(g)$ by Shi's estimates \cite{Shi-Def}.

\end{enumerate}
\end{proof}

\subsection{Spectrum of steady gradient Ricci solitons}
\subsubsection{Spectrum of $\Delta_f$}

We begin by establishing bounds on the bottom of the laplacian $\Delta_f$ on a non trivial steady gradient Ricci soliton. We recall first a useful tool to bound the spectrum of a weighted laplacian on a Riemannian manifold from below as soon as a positive eigenfunction exists : this is a simple modification of lemma $7.6$ of \cite{Lee-Hyp} which uses integration by part only.

\begin{lemma}\label{bd-spec-pos}
 Let $(M,g)$ be a Riemannian manifold endowed with a $C^1$ potential function $f$ on $M$. If there exists a positive $C^2$ function $\phi$ on $M$ such that $\Delta_f\phi\leq-\lambda \phi$ for a real number $\lambda$, then $\lambda_1(\Delta_f)\geq\lambda$.
\end{lemma}

  The following is a simple adaptation of \cite{Mun-Wan} :

\begin{prop}\label{Mun-Wan-Har-Sgs}
Let $(M,g,\nabla f)$ be a steady gradient Ricci soliton. Then, for any compactly supported functions $\phi$ and any $\alpha\in(0,1]$,
\begin{eqnarray*}
\int_{M}\left(\alpha^2R+\lambda(g)\alpha(1-\alpha)\right)\phi^2d\mu_f\leq\int_{M}\arrowvert\nabla\phi\arrowvert^2d\mu_f.
\end{eqnarray*}
\end{prop}

\begin{proof}
Indeed, by the soliton identities \ref{sol-id-sgs},
\begin{eqnarray*}
\Delta_f(e^{-\alpha f})&=&\left(-\alpha\Delta f-\alpha(1-\alpha)\arrowvert\nabla f\arrowvert^2\right)e^{-f}=\left(-\alpha^2\Delta f-\alpha(1-\alpha)\lambda(g)\right)e^{-f}\\
&=&-\alpha(R+(1-\alpha)\lambda(g))e^{-\alpha f}.
\end{eqnarray*}
Then the proposition follows by applying lemma \ref{bd-spec-pos} to $\phi:=e^{-\alpha f}$.
\end{proof}

One can actually prove by hand that the Bryant soliton is strictly stable with the help of proposition \ref{Mun-Wan-Har-Sgs} :

\begin{prop}
Let $n\geq 3$.The $n$-dimensional Bryant soliton is $c$-stable where $c=c(n,\lambda(g))>0$.
\end{prop}

\begin{proof}
As in the proof of proposition \ref{prop-stab-rot-sym}, we have, for $n\geq 3$,
\begin{eqnarray*}
R\arrowvert h\arrowvert^2-2\Rm(h,h)\geq2\sqrt{n-1}(\sqrt{n-1}-1)a\arrowvert h\arrowvert^2=:c_1(n)a\arrowvert h\arrowvert^2.
\end{eqnarray*}
Hence, by the Hardy inequality \ref{Mun-Wan-Har-Sgs} for $\alpha=1$,
\begin{eqnarray*}
\int_{M}\arrowvert \nabla h\arrowvert^2-2\Rm(h,h)d\mu_f&\geq&\int_{M}c_1(n)a\arrowvert h\arrowvert^2d\mu_f,
\end{eqnarray*}
and for $\alpha=1/2$,
\begin{eqnarray*}
\int_{M}\arrowvert \nabla h\arrowvert^2-2\Rm(h,h)d\mu_f&\geq&\int_{M}\left(\frac{\lambda(g)}{4}-c_2(n)\arrowvert\Rm\arrowvert\right)\arrowvert h\arrowvert^2d\mu_f.
\end{eqnarray*}
Outside a compact set $K$, $\frac{\lambda(g)}{4}-c_2(n)\arrowvert\Rm\arrowvert\geq\frac{\lambda(g)}{8}$ and $\inf_{K}a>0$. Define 
\begin{eqnarray*}
A:=\max\left(0,\frac{1}{2}-\frac{\inf_{K}\left(\frac{\lambda(g)}{4}-c_2(n)\arrowvert\Rm\arrowvert\right)}{\inf_{K}c_1(n)a}\right). 
\end{eqnarray*}
Therefore, by the very definition of $A$,
\begin{eqnarray*}
&&\int_{M}\arrowvert \nabla h\arrowvert^2-2\Rm(h,h)d\mu_f\geq\\
&&\frac{1}{1+A}\int_{M}\left(\frac{\lambda(g)}{4}-c_2(n)\arrowvert\Rm\arrowvert+Ac_1(n)a\right)\arrowvert h\arrowvert^2d\mu_f\\
&&\geq\frac{1}{1+A}\int_{K}\left(\inf_{K}\left(\frac{\lambda(g)}{4}-c_2(n)\arrowvert\Rm\arrowvert\right)+A\inf_{K}c_1(n)a\right)\arrowvert h\arrowvert^2d\mu_f\\
&&+\frac{\lambda(g)}{8(1+A)}\int_{M\setminus K}\arrowvert h\arrowvert ^2d\mu_f\\
&&\geq\frac{1}{1+A}\int_{K}\frac{\inf_{K}c_1(n)a}{2}\arrowvert h\arrowvert^2d\mu_f+\frac{\lambda(g)}{8(1+A)}\int_{M\setminus K}\arrowvert h\arrowvert ^2d\mu_f\\
&&\geq c(n,\lambda(g))\int_{M}\arrowvert h\arrowvert^2d\mu_f,
\end{eqnarray*}
where $$c(n,\lambda(g)):=\min\left(\frac{\inf_{K}c_1(n)a}{2(1+A)},\frac{\lambda(g)}{8(1+A)}\right).$$
\end{proof}

We pursue our investigation of the discrete spectrum of $\Delta_f$ on functions. The following identity is a simple adaptation of the arguments of \cite{Don-Gar} :

\begin{lemma}\label{Id-spec}
Let $(M,g)$ be a Riemannian manifold. Let $\phi$, $\psi$ and $f$ be smooth functions on $M$. Then,
\begin{eqnarray*}
\div_f\left(2<\nabla_{\nabla \phi}\psi,\nabla \psi>-\arrowvert\nabla \psi\arrowvert^2\nabla\phi\right)&=&2<\nabla_{\nabla\phi}\psi,\Delta_f\psi>\\
&&+2\nabla^2\phi(\nabla \psi,\nabla \psi)-\arrowvert\nabla \psi\arrowvert^2\Delta_f\phi.
\end{eqnarray*}
\end{lemma}

\begin{proof}
Routine computations.
\end{proof}
The following corollary will not be used for the sequel but is interesting in itself :

\begin{coro}
Let $(M,g,\nabla f)$ be a steady gradient Ricci soliton with nonnegative Ricci curvature, and positive Ricci curvature at one point. Then $\sigma_{disc}(\Delta_f)=\emptyset.$
\end{coro}

\begin{proof}
Assume there exists $\lambda\in\mathbb{R}$ such that $\Delta_f\psi=-\lambda\psi$ with $\psi\in L^2_f$.
Integrate the identity given by \ref{Id-spec} to $\phi=f$, and $\psi$. Then,
\begin{eqnarray*}
\int_{M}2\nabla^2f(\nabla\psi,\nabla\psi)d\mu_f&=&\int_{M}\arrowvert\nabla \psi\arrowvert^2\Delta_ff-2<\nabla f,\nabla\psi>\Delta_f\psi d\mu_f\\
&=&\int_{M}\arrowvert\nabla \psi\arrowvert^2\Delta_ff+2\lambda<\nabla f,\nabla\psi>\psi d\mu_f\\
&=&\int_{M}\arrowvert\nabla \psi\arrowvert^2\Delta_ff+\lambda<\nabla f,\nabla\psi^2> d\mu_f\\
&=&\int_{M}(\arrowvert\nabla \psi\arrowvert^2-\lambda\psi^2)\Delta_ff\\
&=&\lambda(g)\int_{M}\arrowvert\nabla \psi\arrowvert^2-\lambda\psi^2 d\mu_f\\
&=&0.
\end{eqnarray*}
 Therefore, $\nabla\psi=0$ on a neighborhood of a point in $M$. As $g$ and $f$ are analytic by \cite{Ban-Ana}, $\psi$ is analytic too, hence $\psi\equiv0$ everywhere.
\end{proof}

\subsubsection{Spectrum of the Lichnerowicz operator}\label{Spec-Lic-op-sgs}

We first begin by analysing the essential spectrum of the Lichnerowicz operator $L$ when the curvature tends to zero at infinity.

\begin{prop}\label{ess-spec-sgs}
Let $(M,g,\nabla f)$ be a steady gradient Ricci soliton with curvature going to zero at infinity, i.e.
\begin{eqnarray*}
\lim_{+\infty}\arrowvert\Rm(g)\arrowvert=0.
\end{eqnarray*}
Then $\sigma_{ess}(L)\subset\left[\lambda_1(\Delta_f),+\infty\right)$.
\end{prop}

\begin{proof}
The proof is a straightforward adaptation of lemma $4.10$ of \cite{Lee-Hyp}. For the convenience of the reader, we will state its precise statement.  The next claim gives a sufficient condition to ensure the Fredholm property of $L-\lambda \Id$ for $\lambda\in\mathbb{R}$.

\begin{claim}\label{Fred-prop}
$L-\lambda\Id : \,^{f}H_{2}^{2}(M,S^2T^*M)\rightarrow \,^{f}H_{2}^{2}(M,S^2T^*M)$ is Fredholm if there exists a compact set $K\subset M$ and a constant $c>0$ such that the following holds for any $h\in C_0^{\infty}(M\setminus K,S^2T^*M)$ : 
\begin{eqnarray*}
c \,^{f}\|h\|_{0,2}\leq \,^{f}\|(L-\lambda\Id)h\|_{0,2}.
\end{eqnarray*}
\end{claim}

With the claim \ref{Fred-prop} in hand, we can now finish the proof of proposition \ref{ess-spec-sgs}. Indeed, let $\lambda$ be any real number strictly less than $\lambda_1(\Delta_f)$ and let $\epsilon >0$. Take any compact set $K_{\epsilon}\subset M$ such that $\sup_{M\setminus K_{\epsilon}}\arrowvert\Rm(g)\arrowvert\leq \epsilon$. Then, for any $h\in C_0^{\infty}(M\setminus K_{\epsilon},S^2T^*M)$,
\begin{eqnarray*}
\,^{f}<(L-\lambda\Id)h,h>_{0,2}\geq (\lambda_1(\Delta_f)-\lambda -c(n)\epsilon)\,^{f}<h,h>_{0,2},
\end{eqnarray*}
that is, by the Cauchy-Schwarz inequality,
\begin{eqnarray*}
 \,^{f}\|(L-\lambda\Id)h\|_{0,2}\geq (\lambda_1(\Delta_f)-\lambda-\epsilon) \,^{f}\|h\|_{0,2}.
\end{eqnarray*}
Hence the result by invoking the claim \ref{Fred-prop}, if $\epsilon$ is sufficiently small.
\end{proof}

We need now to investigate the discrete spectrum of the Lichnerowicz operator. For that purpose, it is necessary to control the decay of an eigenvector of this operator at infinity. This is achieved with a technique initiated by Agmon \cite{Agm-Dec}. The following theorem is an ad-hoc adaptation of his arguments in the setting of Riemannian manifolds. Nonetheless, the presentation is inspired by \cite{Li-Wan-Poi}.

\begin{theo}(Agmon type estimate)\label{Agmon-steady}
Let $(M,g,\nabla f)$ be a steady gradient Ricci soliton with nonnegative Ricci curvature and bounded curvature normalized such that $\lambda(g)=1$. Assume $\lim_{+\infty}R=0$.

Let $h\in L_f^2(M,S^2T^*M)$ be an eigenvector of  $L$ associated to an eigenvalue $\lambda$ with $\lambda<\lambda_{ess}(L)$, where $\lambda_{ess}(L)$ is the bottom of the essential spectrum of $L$. Then $h$ has exponential decay. More precisely,
\begin{eqnarray*}
e^{\alpha_{\epsilon} f+f/2}h\in L^{\infty}(M,S^2T^*M),
\end{eqnarray*}
where $\alpha_{\epsilon}^2:=\lambda_{ess}(L)-\lambda-\epsilon$, for any $\epsilon$ small enough.
\end{theo}

\begin{proof}
As $\lim_{+\infty}R=0$ and $\Ric\geq 0$, the function $f$ is proper by lemma \ref{gro-pot-fct}. 
\begin{claim}
$e^{\delta\alpha_{\epsilon} f}h\in L^2_f$, for any $\delta\in(0,1)$ and $\epsilon$ positive small enough.
\end{claim}
\begin{proof}
Let $F$ be any smooth function on $M$ to be defined later and let $\phi$ be any Lipschitz function on $M$ with compact support. Then,

\begin{eqnarray*}
&&\int_{M}\arrowvert\nabla(e^F\phi h)\arrowvert^2-2\Rm(g)(e^F\phi h,e^F\phi h)d\mu_f=\\
&&\int_M\arrowvert\nabla(\phi h)\arrowvert^2e^{2F}+\phi^2\arrowvert h\arrowvert^2\arrowvert\nabla F\arrowvert^2e^{2F}+\frac{1}{2}<\nabla(\phi^2 \arrowvert h\arrowvert^2),\nabla e^{2F}>d\mu_f\\
&&-\int_M2\Rm(g)(e^F\phi h,e^F\phi h)d\mu_f\\
&=&-\int_M\langle\phi h,\Delta_f(\phi h)+2<\nabla(\phi h),\nabla F>\rangle e^{2F}+\frac{1}{2}\phi^2\arrowvert h\arrowvert^2\Delta_f(e^{2F})-\phi^2\arrowvert h\arrowvert^2\arrowvert\nabla F\arrowvert^2e^{2F}d\mu_f\\
&&-\int_M2\Rm(g)(e^F\phi h,e^F\phi h)d\mu_f\\
&=&\int_M-\langle\phi h,\Delta_f(\phi h)\rangle e^{2F}+\phi^2 \arrowvert h\arrowvert^2\arrowvert\nabla F\arrowvert^2e^{2F}d\mu_f-\int_M2\Rm(g)(e^F\phi h,e^F\phi h)d\mu_f\\
&=&\int_M(\lambda+\arrowvert\nabla F\arrowvert^2)\phi^2 \arrowvert h\arrowvert^2e^{2F}+\left(<\nabla F,\nabla\phi^2>+\arrowvert\nabla \phi\arrowvert^2\right)\arrowvert h\arrowvert^2e^{2F}d\mu_f.
\end{eqnarray*}
Now, by definition of $\lambda_{ess}(L)$, for any $\epsilon>0$, there exists a compact $K_{\epsilon}\subset M$ such that one has the following inequality, for any $\phi\in C_0^{\infty}(M\setminus K_{\epsilon})$ :
\begin{eqnarray*}
(\lambda_{ess}(L)-\epsilon)\int_M\phi^2 \arrowvert h\arrowvert^2e^{2F}d\mu_f&\leq&\int_M(\lambda+\arrowvert\nabla F\arrowvert^2)\phi^2 \arrowvert h\arrowvert^2e^{2F}d\mu_f\\
&&+\int_M(<\nabla F,\nabla\phi^2>+\arrowvert\nabla \phi\arrowvert^2)\arrowvert h\arrowvert^2e^{2F}d\mu_f.
\end{eqnarray*}

 Fix now a large radius $R_0$ such that  $K_{\epsilon}\subset \{f\leq R_0\}$. Define $\alpha_{\epsilon}^2:=\lambda_{ess}(L)-\epsilon-\lambda>0$, for $\epsilon$ small enough.
 Fix also a number $k$ such that $k\geq (R_0+1)(1+\delta)\alpha_{\epsilon}$ and a radius $R$ larger than $k/{(1+\delta)\alpha_{\epsilon}}$.
Define now on $\{f\geq R_0\}$ :
\[
F =
\left\{
\begin{array}{rl}
\delta\alpha_{\epsilon} f & \mbox{if } \alpha_{\epsilon} f \leq \frac{k}{1+\delta} \\ \\

k-\alpha_{\epsilon} f & \mbox{if } \alpha_{\epsilon} f\geq\frac{k}{1+\delta} 
\end{array}
\right.\quad
\phi =
\left\{
\begin{array}{rl}
 f-R_0 & \mbox{if } R_0\leq f\leq R_0+1 \\ \\
1 & \mbox{if } R_0+1\leq f\leq R\\\\
\frac{2R-f}{R} & \mbox{if} \quad R\leq f\leq 2R\\\\
0 & \mbox{otherwise}.
\end{array}
\right.
\]
Plugging this in the previous inequality lets us to estimate each term :
\begin{eqnarray*}
&&\int_{f\geq R_0}\arrowvert\nabla\phi\arrowvert^2\arrowvert h\arrowvert^2e^{2F}d\mu_f\leq\int_{R_0\leq f\leq R_0+1}\arrowvert\nabla f\arrowvert^2 \arrowvert h\arrowvert^2e^{\delta \alpha_{\epsilon} f}d\mu_f\\
&&+\frac{1}{R^2}\int_{R\leq f\leq 2R}\arrowvert\nabla f\arrowvert^2\arrowvert h\arrowvert^2e^{2(k-\alpha_{\epsilon} f)}d\mu_f,\\
&&\int_{f\geq R_0}\arrowvert\nabla F\arrowvert^2\phi^2\arrowvert h\arrowvert^2e^{2F}d\mu_f\leq\int_{R_0\leq f\leq k/{(1+\delta)\alpha_{\epsilon}}}\delta^2\alpha_{\epsilon}^2\arrowvert\nabla f\arrowvert^2\phi^2 \arrowvert h\arrowvert^2e^{2\delta\alpha_{\epsilon} f}d\mu_f\\
&&+\int_{k/{(1+\delta)\alpha_{\epsilon}}\leq f\leq2R}\alpha_{\epsilon}^2\arrowvert\nabla f\arrowvert^2\phi^2\arrowvert h\arrowvert^2e^{2(k-\alpha_{\epsilon} f)}d\mu_f,\\
&&2\int_{f\geq R_0}<\nabla F,\nabla\phi>\phi \arrowvert h\arrowvert^2e^{2F}d\mu_f\leq2\int_{R_0\leq f\leq R_0+1}\delta\alpha_{\epsilon} \arrowvert\nabla f\arrowvert^2\arrowvert h\arrowvert^2e^{2\delta\alpha_{\epsilon} f }d\mu_f\\
&&+\frac{2\alpha_{\epsilon}}{R}\int_{R\leq f\leq 2R}\arrowvert \nabla f\arrowvert^2\arrowvert h\arrowvert^2e^{2(k-\alpha_{\epsilon} f)}d\mu_f.
\end{eqnarray*}

Since $\arrowvert\nabla f\arrowvert^2\leq \lambda(g)=1$,
\begin{eqnarray*}
&&\int_{R_0\leq f\leq k/{(1+\delta)\alpha_{\epsilon}}}\alpha_{\epsilon}^2(1-\delta^2)\phi^2 \arrowvert h\arrowvert^2e^{2\delta \alpha_{\epsilon} f}d\mu_f\leq\\
 &&C(R_0)+\frac{1}{R^2}\int_{R\leq f\leq 2R}\arrowvert h\arrowvert^2e^{k-\alpha_{\epsilon} f}d\mu_f+\frac{2\alpha_{\epsilon} c(g)}{R}\int_{R\leq f\leq 2R}\arrowvert h\arrowvert^2e^{2(k-\alpha_{\epsilon} f)}d\mu_f.
\end{eqnarray*}
By letting $R$ tend to $+\infty$, we have, for any $k$ sufficiently large,
\begin{eqnarray*}
\int_{R_0\leq f\leq k/{(1+\delta)\alpha_{\epsilon}}}\alpha_{\epsilon}^2(1-\delta^2)\phi^2 \arrowvert h\arrowvert^2e^{2\delta \alpha_{\epsilon} f}d\mu_f\leq C(R_0).
\end{eqnarray*}
Hence the result.
\end{proof}
\begin{claim}
$e^{\alpha_{\epsilon} f}h\in L^2_f$, for $\epsilon$ positive small enough.
\end{claim}
 By the previous claim, one has $e^{\delta\alpha_{\epsilon} f}h\in L^2_f$, for $\epsilon$ positive small enough and $\delta\in[0,1)$. Define $\delta$ to be $1-\epsilon$, then $\delta\alpha_{\epsilon}=\alpha_{N(\epsilon)}$ where $N(\epsilon)$ goes to zero as $\epsilon$ tends to zero.

\begin{claim}
$e^{\alpha_{\epsilon} f+f/2} h\in L^{\infty}(M,S^2T^*M).$
\end{claim}
\begin{proof}
As in Agmon's book \cite{Agm-Dec}, one only needs a local Sobolev inequality for the measure $d\mu_f$ to perform a Moser iteration. As the Ricci curvature is bounded from below, by the results of \cite{Sal-Cos-Uni-Ell}, one has the following local Sobolev inequality : 
\begin{eqnarray*}
\left(\frac{1}{\vol B(x,r)}\int_{B(x,r)}\arrowvert\phi\arrowvert^{\frac{2n}{n-2}}d\mu_g\right)^{\frac{n-2}{n}}\leq \left(\frac{C(r_0)r^2}{\vol B(x,r)}\int_{B(x,r)}\arrowvert\nabla\phi\arrowvert^2+r^{-2}\phi^2d\mu_g\right),
\end{eqnarray*}
for any $\phi\in C_0^{\infty}(B(x,r))$ for all $x\in M$ and $0<r<r_0$.

Now, as the oscillation of $f$ is bounded on $B(x,r)$ by a constant depending only on $r_0$, we have the following local weighted Sobolev inequality :
\begin{eqnarray*}
\left(\frac{1}{\vol_f B(x,r)}\int_{B(x,r)}\arrowvert\phi\arrowvert^{\frac{2n}{n-2}}d\mu_f\right)^{\frac{n-2}{n}}\leq \left(\frac{C(r_0)r^2}{\vol_f B(x,r)}\int_{B(x,r)}\arrowvert\nabla\phi\arrowvert^2+r^{-2}\phi^2d\mu_f\right),
\end{eqnarray*}
for any $\phi\in C_0^{\infty}(B(x,r))$ for all $x\in M$ and $0<r<r_0$, where $\vol_f B(x,r):=\int_{B(x,r)} d\mu_f$.
Now, $\arrowvert h\arrowvert$ satisfies, in the weak sense,
\begin{eqnarray*}
\Delta_f\arrowvert h\arrowvert\geq -c(\lambda,\arrowvert\Rm(g)\arrowvert_{L^{\infty}})\arrowvert h\arrowvert.
\end{eqnarray*}
Therefore, by performing a Moser iteration, we have,
\begin{eqnarray*}
\sup_{B(x,r/2)}\arrowvert h\arrowvert\leq C(\lambda,\arrowvert\Rm(g)\arrowvert_{\infty},r_0)\left(\frac{1}{\vol_fB(x,r)}\int_{B(x,r)}\arrowvert h\arrowvert^2d\mu_f\right)^{1/2}
\end{eqnarray*}
for all $x\in M$ and any $0<r<r_0$. Using the previous claim, we get,
\begin{eqnarray*}
\sup_{B(x,r/2)}e^{\alpha_{\epsilon} f+f/2}\arrowvert h\arrowvert&\leq& C(\lambda,\arrowvert\Rm(g)\arrowvert_{\infty},r_0,\epsilon)\left(\frac{e^{(1-\epsilon/(2\alpha_0))f(x)}}{\vol_fB(x,r)}\int_{B(x,r)}\arrowvert e^{\alpha_{\epsilon/2} f}h\arrowvert^2d\mu_f\right)^{1/2}\\
&\leq&C(\lambda,\arrowvert\Rm(g)\arrowvert_{\infty},r_0,\epsilon)\frac{e^{(1-\epsilon/(2\alpha_0))f(x)}}{\vol_fB(x,r)}\| e^{\alpha_{\epsilon/2} f}h\|_{L^2_f}.
\end{eqnarray*}

Now, by the Bishop-Gromov comparison theorem, as the Ricci curvature is non negative,

\begin{eqnarray*}
\vol_f B(x,r)&\geq& C(r)e^{f(x)} \vol B(x,r)\\
&\geq& C(r)e^{f(x)}\frac{\vol B(x,r)}{\vol B(x, r_p(x)+r)}\vol B(p,r)\\
&\geq& C(p,r)e^{f(x)}\left(\frac{r}{r+r_p(x)}\right)^n,
\end{eqnarray*}
for some point $p\in M$.
Since $f\geq cr_p(x)$ outside a compact set, we have 
\begin{eqnarray*}
\sup_{x\in M}\frac{e^{(1-\epsilon/(2\alpha_0))f(x)}}{\vol_fB(x,r)}\leq C(p,r,\epsilon,\alpha_0),
\end{eqnarray*}
 for some fixed point $p\in M$.
 Hence the exponential decay of $h$.

\end{proof}
\end{proof}

We are now in a position to prove theorem  \ref{spec-sgs-pos-cur} :

\begin{proof}[Proof of theorem \ref{spec-sgs-pos-cur}]
Let $h\in L_f^2(M,S^2T^*M)$ be an eigenvector of  $L$ associated to an eigenvalue $-\lambda$ with $\lambda\geq 0$. We claim that $h=0$.

\begin{claim}
$\sup_{M_{\leq t}}\arrowvert h\arrowvert\leq \frac{\lambda(g)}{\inf_{M_t}\Ric}\sup_{M_t}\arrowvert h\arrowvert.$
\end{claim}
\begin{proof}
The following proof is inspired by Brendle's argument in the proof of the Perelman's conjecture \cite{Bre-Rot-3d}. Indeed, since $\Delta_fh+2\Rm(g)\ast h=\lambda h$ and $\Li_{\nabla f}(g)\in\ker L$, we have for any $\theta\in\mathbb{R}$,
$\Delta_f h_{\theta}+2\Rm(g)\ast h_{\theta}=\lambda h_{\theta}+\theta\lambda \Li_{\nabla f}(g)$, where $h_{\theta}:=h-\theta\Li_{\nabla f}(g)$ and where $\theta$ is a real number.

Fix $t>\min_{M} f$. Define $\theta_t:=\sup\{\theta/h_{\theta}\arrowvert_{M_{\leq t}}\geq 0\}$. Since $f$ is proper, $M_{\leq t}:=\{f\leq t\}$ is compact. Therefore $\theta_t$ is finite and there exists a point $x\in M_{\leq t}$ and a unitary vector $v\in T_xM$ such that $h_{\theta_t}(x)(v,v)=0.$ \\

\begin{enumerate}
\item If $x\in M_t:=\{f=t\}$ then $h_{\theta_t}(x)(v,v)=0=h(x)(v,v)-2\theta_t\Ric(x)(v,v)$. In particular, 
\begin{eqnarray*}
\arrowvert\theta_t\arrowvert&\leq&\frac{\sup_{M_t}\arrowvert h\arrowvert}{2\inf_{M_t}\Ric},\\
h&\geq& -\frac{\sup_{M_t}\arrowvert h\arrowvert}{2\inf_{M_t}\Ric}\Li_{\nabla f}(g)\geq-\frac{\lambda(g)}{\inf_{M_t}\Ric}\sup_{M_t}\arrowvert h\arrowvert g\\
%&\geq& -(2\sup_{M_t} R)\sup_{M_t}\arrowvert h\arrowvert g\quad\mbox{on $M_{\leq t}$.}
\end{eqnarray*}

\item If $x\in M_{<t}$ and $\theta_t\geq 0$, then $h\geq 0$ on $M_{\leq t}$.

If $\theta_t\leq 0$, define $g(\tau):=\phi_{\tau}^*g$ and $h_{\theta_t}(\tau):=\phi_{\tau}^*h_{\theta_t}$ where $\phi_{\tau}$ is the flow generated by $-\nabla f$, then 
\begin{eqnarray*}
\partial_{\tau}h_{\theta_t}(\tau)&=&\Delta_{L,g(\tau)}h_{\theta_t}(\tau)-\lambda h_{\theta_t}(\tau)-\lambda\theta_t \phi_{\tau}^*\Li_{\nabla f}(g)\\
&\geq&\Delta_{L,g(\tau)}h_{\theta_t}(\tau)-\lambda h_{\theta_t}(\tau),
\end{eqnarray*}
where $\Delta_{L,g(\tau)}$ is the Lichnerowicz laplacian associated to the metric $g(\tau)$.

 There exists a neighborhood $U_x$ of $x$ and a time $T_x>0$ such that $$(U_x,g(\tau))_{\tau\in[-T_x,T_x]}\subset (M_{<t},g),$$ so that on $(U_x,g(\tau))_{\tau\in[-T_x,T_x]}$, $h_{\theta_t}(\tau)\geq 0$. Since the sectional curvature is non negative, one can adapt a local version of Hamilton's maximum principle for system [Theorem $12.50$, \cite{Cho-Lu-Ni-II}] to establish that $\ker h_{\theta_t}$ is a smooth distribution invariant by parallel translation in the case where $\theta_t\leq 0$. As we assume $\Ric>0$, the manifold cannot split. In particular, $h_{\theta_t}\equiv 0$ on $U_x$. As $M_{<t}$ is totally geodesic, $h_{\theta_t}\equiv 0$ on $M_{< t}$. In particular, for any $s<t$,
\begin{eqnarray*}
\arrowvert\theta_t\arrowvert\leq\frac{\sup_{M_s}\arrowvert h\arrowvert}{2\inf_{M_s}\Ric}.
\end{eqnarray*}
If $s$ tends to $t$, one has,
\begin{eqnarray*}
h\geq-\frac{\lambda(g)}{\inf_{M_t}\Ric}\sup_{M_t}\arrowvert h\arrowvert g\quad\mbox{on $M_{\leq t}$.}
\end{eqnarray*}

The same reasoning as above applied to $-h$ gives the estimate.
\end{enumerate}
\end{proof}
Without loss of generality, one can normalize the metric so that $\lambda(g)=1$ : the assumptions made in theorem \ref{spec-sgs-pos-cur} are invariant under scaling. Now, using proposition \ref{Mun-Wan-Har-Sgs} with the normalization $\lambda(g)=1$ gives $\lambda_{ess}(L)\geq 1/4$, therefore the exponential decay of $h$ given by theorem \ref{Agmon-steady} is $e^{-(1-\epsilon)f}$ for any $\epsilon$ positive small enough since $h$ is an eigenvector associated to an non positive eigenvalue $\lambda$. Now, we use the previous claim together with the assumption on the decay of the Ricci curvature to get $h=0$.
\end{proof}

\subsection{Spectrum of Expanding gradient Ricci solitons}

\subsubsection{Spectrum of $\Delta_f$}

This section is devoted to establish bounds on the bottom of the laplacian $\Delta_f$ on a non trivial expanding gradient Ricci soliton. The following is due to \cite{Mun-Wan} :

\begin{prop}\label{Mun-Wan-Har-Egs}(Munteanu-Wang)
Let $(M,g,\nabla f)$ be an expanding gradient Ricci soliton. Then, for any compactly supported functions $\phi$,
\begin{eqnarray*}
\int_{M}\left(R+\frac{n}{2}\right)\phi^2d\mu_f\leq\int_{M}\arrowvert\nabla \phi\arrowvert^2d\mu_f.
\end{eqnarray*}
\end{prop}

\begin{proof}
The idea is to note that $\Delta_f(e^{-f})=-\left(R+\frac{n}{2}\right)e^{-f}$ and then one applies lemma \ref{bd-spec-pos} to $\phi:=e^{-f}$.
\end{proof}
\begin{rk}
Proposition \ref{Mun-Wan-Har-Egs} shows that $\lambda_1(\Delta_f)\geq n/2$ for any expanding gradient Ricci soliton with non negative scalar curvature. Moreover, one can even show that $\lambda_1(\Delta_f)=n/2$ for an expanding gradient Ricci soliton with non negative scalar curvature if and only if it is isometric to the Euclidean space. 
\end{rk}
As in the previous section, one can actually prove by hand that the Bryant expanding gradient Ricci solitons are strictly stable with the help of proposition \ref{Mun-Wan-Har-Egs} :
\begin{prop}\label{prop-stab-rot-sym}
Any rotational symmetric positively curved expanding gradient Ricci soliton satisfies $\lambda_1(\Delta_f)\geq n/2$ for $n\geq 2$.
\end{prop}

\begin{proof}
Let $a$ denote the radial sectional curvatures and let $b$ denote the spherical curvatures. Then,
\begin{eqnarray*}
R=2(n-1)a+(n-1)(n-2)b,
\end{eqnarray*}
and, if $h$ is a symmetric $2$-tensor,
\begin{eqnarray*}
\Rm(h,h):=\Rm_{ijji}h_{ii}h_{jj}&=&2ah_{\partial r\partial r}\sum_{i\neq \partial r}h_{ii}+b\sum_{i\neq j;i\neq\partial r;j\neq\partial r}h_{ii}h_{jj}\\
&=&2ah_{\partial r\partial r}\sum_{i\neq \partial r}h_{ii}+b\left(\left(\sum_{i\neq\partial r}h_{ii}\right)^2-\sum_{i\neq\partial r}h_{ii}^2\right)\\
&\leq&2a\arrowvert h_{\partial r\partial r}\arrowvert\sqrt{n-1}\left(\sum_{i\neq\partial r}h_{ii}^2\right)^{1/2}+(n-2)b\sum_{i\neq\partial r}h_{ii}^2\\
&\leq&\sqrt{n-1}a\arrowvert h\arrowvert^2+(n-2)b\sum_{i\neq\partial r}h_{ii}^2.
\end{eqnarray*}
Finally,
\begin{eqnarray*}
R\arrowvert h\arrowvert^2-2\Rm(h,h)&\geq&(2(n-1)a+(n-1)(n-2)b)\arrowvert h\arrowvert^2\\
&&-2(\sqrt{n-1}a\arrowvert h\arrowvert^2+(n-2)b\sum_{i\neq\partial r}h_{ii}^2)\\
&\geq&2\sqrt{n-1}(\sqrt{n-1}-1)a\arrowvert h\arrowvert^2+(n-2)(n-3)b\sum_{i\neq\partial r}h_{ii}^2\\
&\geq& 0.
\end{eqnarray*}
Therefore, by the Hardy inequality \ref{Mun-Wan-Har-Egs},
\begin{eqnarray*}
\int_{M}\arrowvert \nabla h\arrowvert^2-2\Rm(h,h)d\mu_f&\geq&\int_{M}\left(R+\frac{n}{2}\right)\arrowvert h\arrowvert^2-2\Rm(h,h)d\mu_f\\
&\geq&\frac{n}{2}\int_{M}\arrowvert h\arrowvert^2d\mu_f.
\end{eqnarray*}

\end{proof}

%\begin{rk}
%If the sectional curvatures are all negative, as for the Bryant expanding gradient Ricci solitons asymptotic to $(C(\mathbb{S}^{n-1}),dr^2+(\alpha r)^2g_{\mathbb{S}^{n-1}})$, with $\alpha\in(1,+\infty)$, the Lichnerowicz operator is also strictly stable since the potential term $-2\Rm(g)(h,h)$ is non negative : see also next section.
%\end{rk}

\subsubsection{Spectrum of the Lichnerowicz operator of expanding gradient Ricci solitons}

Now, we study both the essential and discrete spectrum of expanding gradient Ricci solitons. The crucial obervation is that $L$ is unitarily conjugate to an harmonic oscillator, i.e. an operator of the type $-\Delta_g+V(g)\ast$ where $V(g)$ is a potential quadratic in the distance to a fixed point of the manifold. We begin with the essential spectrum :

\begin{prop}
Let $(M,g,\nabla f)$ be an expanding gradient Ricci soliton satisfying $(\Hyp)$.
Then $\sigma_{ess}(L)=\emptyset$.
%\subset\left[\max\left(\lambda_1(\Delta_f),\frac{1}{2}\right)+\infty\right)$.
\end{prop}

\begin{proof}
Let $h\in C^{\infty}(M,S^2T^*M)$. Then, 
\begin{eqnarray*}
e^{f/2}L(e^{-f/2}h)&=&-\Delta_gh+\frac{1}{2}\left(\Delta_g f+\frac{\arrowvert\nabla f\arrowvert^2}{2}\right)h-2\Rm(g)\ast h\\
&=:&-\Delta_gh+V(g)\ast h.
\end{eqnarray*}
Now, as $V(g)$ is equivalent at infinity to $r_p^2(\cdot)$ for $p\in M$, we claim, by using the proof of proposition \ref{ess-spec-sgs} (see also [Theorem $XIII.67$ ; \cite{Ree-Sim-IV}]), that $\sigma_{ess}(-\Delta_g+V(g)\ast)=\emptyset$ and so is $\sigma_{ess}(L)$. 
%This follows from the same argument as in the steady case plus the fact that 
%\begin{eqnarray*}
%2<Lh,h>_{L^2_f}&=&<\Cod(h),\Cod(h)>_{L^2_f}+2<\div_f(h),\div_f(h)>_{L^2_f}\\
%&&-2<\Ric_f(h),h>_{L^2_f}-2<\Rm(g)\ast h,h>_{L^2_f}\\
%&\geq&(1-\epsilon)<h,h>_{L^2_f},
%\end{eqnarray*}
%for any $h\in C_0^{\infty}(M\setminus K_{\epsilon},S^2T^*M)$, where $K_{\epsilon}\subset M$ is a compact.

\end{proof}

%\subsubsection{Exponential decay of eigenfunctions of $\Delta_f$}
As in section \ref{Spec-Lic-op-sgs}, to investigate the discrete spectrum of the Lichnerowicz operator, it is necessary to control the decay of an eigenvector of this operator at infinity. The following theorem is an adaptation of Agmon's \cite{Agm-Dec} and Barry Simon's \cite{Sim-Dec}  arguments in the setting of Riemannian manifolds. Again, the presentation is inspired by \cite{Li-Wan-Poi}.

\begin{theo}\label{Agmon-expanding}
Let $(M,g)$ be a complete Riemannian manifold. Let $V$ be a potential acting on a geometric bundle $E$ over $M$ such that there exists a smooth proper function $W$ and $\mu\in(1,+\infty)$ such that
$\mu^2\arrowvert\nabla W\arrowvert^2\leq-\Delta+V$ outside a compact set, in the sense of quadratic forms. Assume that $\arrowvert \nabla W\arrowvert$ is a proper function and 
\begin{eqnarray*}
\limsup_{+\infty}\arrowvert\nabla W\arrowvert e^{-W/2}<+\infty.
\end{eqnarray*}

Let  $h$ be an $L^2$-eigenfunction of $-\Delta+V$ with eigenvalue $\lambda\in\sigma_{disc}(-\Delta+V)$. Then,
\begin{eqnarray*}
e^{\delta W}h\in L^2,
\end{eqnarray*}
for any $\delta\in[0,\mu).$
\end{theo}

\begin{proof}
Since $W$ and $\arrowvert \nabla W\arrowvert$ are proper, one can assume that $M$ has one end, without loss of generality.

Let $F$ be any smooth function on $M$ to be defined later and let $\phi$ be any Lipschitz function on $M$ with compact support. Then,

\begin{eqnarray*}
&&\int_{M}\arrowvert\nabla(e^F\phi h)\arrowvert^2+<V\ast e^F\phi h,\phi h>d\mu_g=\\
&&\int_M\arrowvert\nabla(\phi h)\arrowvert^2e^{2F}+\arrowvert\phi h\arrowvert^2\arrowvert\nabla F\arrowvert^2e^{2F}d\mu_g+\int_M\frac{1}{2}<\nabla\arrowvert\phi h\arrowvert^2,\nabla e^{2F}>+<V\ast(e^F\phi h),\phi h>d\mu_g\\
&&=-\int_M\left<\phi h,\Delta(\phi h)+2<\nabla(\phi h),\nabla F>\right>e^{2F}-\frac{1}{2}\arrowvert\phi h\arrowvert^2\Delta(e^{2F})+\arrowvert\phi h\arrowvert^2\arrowvert\nabla F\arrowvert^2e^{2F}d\mu_g\\
&&+\int_M<V\ast e^F\phi h,\phi h>d\mu_g\\
&&=\int_M-<\phi h,\Delta(\phi h)e^{2F}>+\arrowvert\phi h\arrowvert^2\arrowvert\nabla F\arrowvert^2e^{2F}d\mu_g+\int_M<V\ast e^F\phi h,\phi h>d\mu_g\\
&&=\int_M(\lambda+\arrowvert\nabla F\arrowvert^2)\arrowvert\phi h\arrowvert^2e^{2F}d\mu_g+\int_M(<\nabla F,\nabla\phi^2>+\arrowvert\nabla \phi\arrowvert^2)\arrowvert h\arrowvert^2e^{2F}d\mu_g.
\end{eqnarray*}
Now, by assumption on $-\Delta +V$, one has the following inequality :
\begin{eqnarray*}
\int_M\mu^2\arrowvert\nabla W\arrowvert^2\arrowvert\phi h\arrowvert^2e^{2F}d\mu_g&\leq&\int_M(\lambda+\arrowvert\nabla F\arrowvert^2)\arrowvert\phi h\arrowvert^2e^{2F}d\mu_g\\
&&+\int_M(<\nabla F,\nabla\phi^2>+\arrowvert\nabla \phi\arrowvert^2)\arrowvert h\arrowvert^2e^{2F}d\mu_g.
\end{eqnarray*}

Take any $\delta\in[0,\mu)$. Fix now a large radius $R_0$ such that  
\begin{eqnarray*}
\liminf_{r_p(x)\geq R_0+1}(\mu^2-\max\{\delta^2,1\})\arrowvert\nabla W\arrowvert^2-\lambda>0.
\end{eqnarray*}
 Fix also a number $K$ such that $K\geq (R_0+1)(1+\delta)$ and a radius $R$ larger than $K/(1+\delta)$.
Define now on $\{W\geq R_0\}$ :
\[
F =
\left\{
\begin{array}{rl}
\delta W & \mbox{if } W \leq \frac{K}{1+\delta} \\ \\

K-W & \mbox{if } W\geq\frac{K}{1+\delta} 
\end{array}
\right.\quad
\phi =
\left\{
\begin{array}{rl}
 W-R_0 & \mbox{if } R_0\leq W\leq R_0+1 \\ \\
1 & \mbox{if } R_0+1\leq W\leq R\\\\
\frac{2R-W}{R} & \mbox{if} \quad R\leq W\leq 2R\\\\
0 & \mbox{otherwise}.
\end{array}
\right.
\]
Plugging this in the previous inequality lets us to estimate each term :
\begin{eqnarray*}
\int_{W\geq R_0}\arrowvert\nabla\phi\arrowvert^2h^2e^{2F}d\mu_g&\leq&\int_{R_0\leq W\leq R_0+1}\arrowvert\nabla W\arrowvert^2 h^2e^{\delta W}d\mu_g\\
&&\frac{1}{R^2}\int_{R\leq W\leq 2R}\arrowvert\nabla W\arrowvert^2\arrowvert h\arrowvert^2e^{K-W}d\mu_g\\
\int_{W\geq R_0}\arrowvert\nabla F\arrowvert^2\arrowvert\phi h\arrowvert^2e^{2F}d\mu_g&\leq&\int_{R_0\leq W\leq K/{1+\delta}}\delta^2\arrowvert\nabla W\arrowvert^2\arrowvert\phi h\arrowvert^2e^{2F}d\mu_g\\
&&+\int_{K/{1+\delta}\leq W\leq2R}\arrowvert\nabla W\arrowvert^2\arrowvert\phi h\arrowvert^2e^{2F}d\mu_g\\
2\int_{W\geq R_0}<\nabla F,\nabla\phi>\phi \arrowvert h\arrowvert^2e^{2F}d\mu_g&\leq&2\int_{R_0\leq W\leq R_0+1}\delta \arrowvert\nabla W\arrowvert^2\arrowvert h\arrowvert^2e^{2F}d\mu_g\\
&&+\frac{2}{R}\int_{R\leq W\leq 2R}\arrowvert \nabla W\arrowvert^2\arrowvert h\arrowvert^2e^{K-W}d\mu_g.
\end{eqnarray*}

Hence,
\begin{eqnarray*}
\int_{R_0\leq W\leq K/{1+\delta}}((\mu^2-\delta^2)\arrowvert\nabla W\arrowvert^2-\lambda)\arrowvert\phi h\arrowvert^2e^{2\delta W}d\mu_g\\
+\int_{K/{1+\delta}\leq W\leq 2R}((\mu^2-1)\arrowvert\nabla W\arrowvert^2-\lambda)\arrowvert\phi h\arrowvert^2e^{2F}d\mu_g&\leq& \\
C(R_0)+\frac{1}{R^2}\int_{R\leq W\leq 2R}\arrowvert\nabla W\arrowvert^2\arrowvert h\arrowvert^2e^{K-W}d\mu_g\\
+\frac{2}{R}\int_{R\leq W\leq 2R}\arrowvert \nabla W\arrowvert^2\arrowvert h\arrowvert^2e^{K-W}d\mu_g.
\end{eqnarray*}
By assumption on the growth of $\arrowvert\nabla W\arrowvert$, by letting $R$ tend to $+\infty$, we have, for any $K$ sufficiently large,
\begin{eqnarray*}
\int_{R_0\leq W\leq K/{1+\delta}}((\mu^2-\delta^2)\arrowvert\nabla W\arrowvert^2-\lambda)\arrowvert\phi h\arrowvert^2e^{2\delta W}d\mu_g\leq C(R_0).
\end{eqnarray*}
Hence the result.
\end{proof}

\begin{coro}\label{int-est-eigen-egs}
Let $(M,g,\nabla f)$ be an expanding gradient Ricci soliton satisfying $(\Hyp)$. Let $h$ be an eigenfunction of $L$ associated to an eigenvalue $\lambda\in \sigma_{disc}(L)$. Then 
\begin{eqnarray*}
e^{\alpha f/2}h\in L^2_f
\end{eqnarray*}
for any $\alpha\in[0,1)$.  
\end{coro}

\begin{proof}
If $h$ is such an eigenfunction then it satisfies : 
\begin{eqnarray*}
-\Delta_g(e^{f/2}h)+\frac{1}{2}\left(\Delta_g f+\frac{\arrowvert\nabla f\arrowvert^2}{2}\right)(e^{f/2}h)-2\Rm(g)\ast h=\lambda (e^{f/2}h).
\end{eqnarray*}
Now, if 
\begin{eqnarray*}
V(g):=\frac{1}{2}\left(\Delta f+\frac{\arrowvert\nabla f\arrowvert^2}{2}\right)-2\Rm(g)\ast,
\end{eqnarray*}
then, by the soliton identities, $-\Delta_g+V(g)\geq \mu^2\arrowvert \nabla W\arrowvert^2$ outside a compact set, where $\mu W:=f/2$. By lemma \ref{sol-id-egs}, both $f$ and $\nabla f$ are proper. Therefore, by the proposition \ref{Agmon-expanding}, we get that $e^{\delta W}(e^{f/2}h)=e^{(\delta/\mu) f/2}(e^{f/2}h)\in L^2,$ for any $\delta/\mu\in[0,1)$. Hence the result. 

\end{proof}

\begin{coro}\label{amorce-decay}
Let $(M,g,\nabla f)$ be an expanding gradient Ricci soliton satisfying $(\Hyp)$. Let $h$ be an eigenfunction of $L$ associated to an eigenvalue $\lambda\in \sigma_{disc}(L)$. Then, 
\begin{eqnarray*}
\exp\left( \alpha f\right)h\in L^{\infty}(M,\mathbb{R}),
\end{eqnarray*}
for any $\alpha\in[0,1)$.  
\end{coro}

\begin{proof}
If $u$ is such an eigenfunction then it satisfies : 
\begin{eqnarray*}
-\Delta(e^{f/2}h)+\frac{1}{2}\left(\Delta f+\frac{\arrowvert\nabla f\arrowvert^2}{2}\right)(e^{f/2}h)=\lambda (e^{f/2}h).
\end{eqnarray*}
In particular, by the Kato inequality,
\begin{eqnarray*}
\Delta(\arrowvert v\arrowvert)\geq -\lambda\arrowvert v\arrowvert,
\end{eqnarray*}
where $v$ does not vanish, with $v:=e^{f/2}h$, outside a compact set.
Since the Ricci curvature is bounded from below by $-K$ with $K$ non negative, we can apply the local Sobolev inequality proved in \cite{Sal-Cos-Uni-Ell} to perform a Moser iteration and get :
\begin{eqnarray*}
\sup_{B(x,r/2)}\arrowvert v\arrowvert\leq C(r_0,\lambda)\left(\avi_{B(x,r)}v^2d\mu_g\right)^{1/2},
\end{eqnarray*}
for any $x\in M$ and any $0<r<r_0$.
Now, by the Bishop-Gromov theorem,
\begin{eqnarray*}
\vol B(x,r)&\geq& c(n,K)e^{-c(n,K)r_p(x)}\vol B(x,r+r_p(x))\\
&\geq& c(n,K)e^{-c(n,K)r_p(x)}\vol B(p,r)\geq c(n,K,p,r)e^{-c(n,K)\sqrt{f(x)}}.
\end{eqnarray*}

Therefore, if $\epsilon>0$ and $\alpha\in[0,1)$, then, by corollary \ref{int-est-eigen-egs},

\begin{eqnarray*}
&&\sup_{B(x,r/2)}e^{(1-\epsilon)\alpha f}\arrowvert v\arrowvert\leq Ce^{c(n,K)\sqrt{f(x)}} \| e^{\alpha f} v\|_{L^2}e^{\alpha(1-\epsilon)\sup_{B(x,r)}f-\alpha\inf_{B(x,r)}f},\\
\end{eqnarray*}
where $C:=C(n,K,p,r,r_0,\lambda)$.
Finally, observe that,
\begin{eqnarray*}
(1-\epsilon)\sup_{B(x,r)}f-\inf_{B(x,r)}f\leq(2+\epsilon)\sqrt{f(x)}-\epsilon f(x)\leq -\frac{\epsilon}{2}f(x)+c(\epsilon),
\end{eqnarray*}
since $\arrowvert\nabla(2\sqrt{f})\arrowvert\leq 1$ on $M$ by the soliton identities \ref{sol-id-egs}.

Since $f$ behaves like $r_p^2$, we have :  $e^{\alpha(1-\epsilon)f}v\in L^{\infty}(M,\mathbb{R}).$

\end{proof}

\begin{rk}
The previous corollary is in sharp contrast with the shrinking case : in that case, at least heuristically, the eigenfunctions behave like polynomials since the weighted laplacian is the Ornstein-Uhlenbeck operator. 
\end{rk}

\subsubsection{Proof of theorem \ref{lic-op-dis-pos-curv}}

Now, we are in a position to prove theorem \ref{lic-op-dis-pos-curv} :

\begin{proof}[Proof of theorem \ref{lic-op-dis-pos-curv}]
Let $h\in L_f^2(M,S^2T^*M)$ be an eigenvector of  $L$ associated to an eigenvalue $-\lambda$ with $\lambda\geq 0$. We claim that $h=0$.

\begin{claim}
$\sup_{M_{\leq t}}\arrowvert h\arrowvert\leq (1+2\sup_{M_{\leq t}}R)\sup_{M_t}\arrowvert h\arrowvert.$
\end{claim}
\begin{proof}
Indeed, since $\Delta_fh+2\Rm(g)\ast h=\lambda h$ and $\Li_{\nabla f}(g)\in\ker L$, we have for any $\theta_t\in\mathbb{R}$,
$\Delta_f h_{\theta_t}+2\Rm(g)\ast h_{\theta_t}=\lambda h_{\theta_t}+\theta_t\lambda \Li_{\nabla f}(g)$, where $h_{\theta_t}=\lambda h_{\theta_t}+\theta_t\lambda \Li_{\nabla f}(g)$.

Fix $t>\min_{M} f$. Define $\theta_t:=\sup\{\theta/h_{\theta}:=h-\theta\Li_{\nabla f}(g)\arrowvert_{M_{\leq t}}\geq 0\}$. Since $f$ is proper, $M_{\leq t}$ is compact. Therefore $\theta_t$ is finite and there exists a point $x\in M_{\leq t}$ and a unitary vector $v\in T_xM$ such that $h_{\theta_t}(x)(v,v)=0.$ \\

\begin{enumerate}

\item If $x\in M_t$ then $h_{\theta_t}(x)(v,v)=0=h(x)(v,v)-\theta_t(1+2\Ric(x)(v,v))$. In particular, 
\begin{eqnarray*}
\arrowvert\theta_t\arrowvert&\leq&\frac{\sup_{M_t}\arrowvert h\arrowvert}{1+2\inf_{M_t}\Ric},\\
h&\geq& -\frac{\sup_{M_t}\arrowvert h\arrowvert}{1+2\inf_{M_t}\Ric}\Li_{\nabla f}(g)\geq-(1+2\sup_{M_{\leq t}}R)\sup_{M_t}\arrowvert h\arrowvert g\\
%&\geq& -(2\sup_{M_t} R)\sup_{M_t}\arrowvert h\arrowvert g\quad\mbox{on $M_{\leq t}$.}
\end{eqnarray*}

\item If $x\in M_{<t}$ and $\theta_t\geq 0$, then $h\geq 0$ on $M_{\leq t}$.

If $\theta_t\leq 0$, define $g(\tau):=(1+\tau)\phi_{\tau}^*g$ and $h_{\theta_t}(\tau):=(1+\tau)\phi_{\tau}^*h_{\theta_t}$ where $\phi_{\tau}$ is the flow generated by $-\nabla f/(1+\tau)$, then, if $\theta_t\leq 0$, 
\begin{eqnarray*}
\partial_{\tau}h_{\theta_t}(\tau)&=&\Delta_{L,g(\tau)}h_{\theta_t}(\tau)-\frac{\lambda h_{\theta_t}(\tau)}{1+\tau}-\lambda\theta_t \phi_{\tau}^*\Li_{\nabla f}(g)\\
&\geq&\Delta_{L,g(\tau)}h_{\theta_t}(\tau)-\frac{\lambda h_{\theta_t}(\tau)}{1+\tau}.
\end{eqnarray*}

 There exists a neighborhood $U_x$ of $x$ and a time $T_x>0$ such that $$(U_x,g(\tau))_{\tau\in[-T_x,T_x]}\subset (M_{<t},g),$$ so that on $(U_x,g(\tau))_{\tau\in[-T_x,T_x]}$, $h_{\theta_t}(\tau)\geq 0$. Since the sectional curvature is non negative, one can adapt a local version of Hamilton's maximum principle for system [Theorem $12.50$, \cite{Cho-Lu-Ni-II}] to establish that $\ker h_{\theta_t}$ is a smooth distribution invariant by parallel translation in the case where $\theta_t\leq 0$. As we assume $\Ric>0$, the manifold cannot split. In particular, $h_{\theta_t}\equiv 0$ on $U_x$. As $M_{<t}$ is totally geodesic, $h_{\theta_t}\equiv 0$ on $M_{< t}$. In particular, for any $s<t$,
\begin{eqnarray*}
\arrowvert\theta_t\arrowvert\leq\sup_{M_s}\arrowvert h\arrowvert.
\end{eqnarray*}
If $s$ tends to $t$, one has,
\begin{eqnarray*}
h\geq-(1+2\sup_{M_{\leq t}}R)\sup_{M_t}\arrowvert h\arrowvert g\quad\mbox{on $M_{\leq t}$.}
\end{eqnarray*}

The same reasoning as above applied to $-h$ gives the estimate.
\end{enumerate}
\end{proof}
Now, a simple modification of the arguments of the proof of corollary \ref{amorce-decay} applied to symmetric $2$-tensors implies an exponential decay of $h$. By the previous claim, we get $h=0$.
\end{proof}

\subsubsection{Proof of theorems \ref{lic-op-dis-pos-scal-curv} and \ref{lic-op-dis-Bochner}}

First, we turn to the proof of theorem \ref{lic-op-dis-pos-scal-curv}.

We need an a priori estimate for symmetric $2$-tensors satisfying $Lh=\lambda h$. This is a simple adaptation of an argument due to Anderson and Chow which is very particular to the structure of the Riemann tensor in dimension $3$.

\begin{theo}[Anderson-Chow]\label{And-Cho-est}
Let $(M^3,g,\nabla f)$ be an expanding gradient Ricci soliton with positive scalar curvature satisfying $(\Hyp)$. Let $h$ be a symmetric $2$-tensor on $M^3$ such that $\Delta_f(h)+2\Rm(g)\ast h=-\lambda h$ for some real number $\lambda\leq 1$. Then,
\begin{eqnarray*}
\sup_{f\leq t}\left(\frac{\arrowvert h\arrowvert}{R}\right)=\sup_{f=t}\left(\frac{\arrowvert h\arrowvert}{R}\right),
\end{eqnarray*}
for $t$ large enough.
\end{theo}

\begin{proof}
Indeed, by adapting the computations of \cite{And-Cho}, one has,
\begin{eqnarray*}
\Delta_f\left(\frac{\arrowvert h\arrowvert^2}{R^2}\right)&=&\frac{2}{R^4}\arrowvert R\nabla h-\nabla R h\arrowvert^2+2(-\lambda+1)\left(\frac{\arrowvert h\arrowvert^2}{R^2}\right)-2\left\langle\nabla\ln R,\nabla\left(\frac{\arrowvert h\arrowvert^2}{R^2}\right)\right\rangle\\
&&+4\left(\arrowvert h\arrowvert^2\arrowvert\Ric\arrowvert^2-2R\tr h<\Ric,h>+2R<\Ric,h^2>+\frac{R^2}{2}\left((\tr h)^2-\arrowvert h\arrowvert^2\right)\right).
\end{eqnarray*}
In \cite{And-Cho}, it is shown that the last term is always nonnegative for any symmetric $2$-tensor $h$ in dimension $3$. Therefore,
\begin{eqnarray*}
\Delta_f\left(\frac{\arrowvert h\arrowvert^2}{R^2}\right)&\geq&2(-\lambda+1)\left(\frac{\arrowvert h\arrowvert^2}{R^2}\right)-2\left\langle\nabla\ln R,\nabla\left(\frac{\arrowvert h\arrowvert^2}{R^2}\right)\right\rangle\\
&\geq&-2\left\langle\nabla\ln R,\nabla\left(\frac{\arrowvert h\arrowvert^2}{R^2}\right)\right\rangle.\\
\end{eqnarray*}
The assumption $(\Hyp)$ is only used to ensure the properness of the potential function $f$. By applying the maximum principle to the domain $\{f\leq t\}$ for sufficiently large $t$ such that $\{f=t\}$ is a smooth compact hypersurface, we get the result.
\end{proof}

\begin{proof}[Proof of theorem \ref{lic-op-dis-pos-scal-curv}]
We will actually prove that $\sigma_{disc}(L)\subset(1,+\infty)$. Indeed, if $h$ is an eigenvector in $L^2_f$ of $L$ associated to an eigenvalue $\lambda\leq 1$, then on the one hand, by corollary \ref{amorce-decay}, $\exp\left( \alpha f\right)h\in L^{\infty}(M,\mathbb{R})$ for any $\alpha\in[0,1)$  and on the other hand, by theorem \ref{And-Cho-est},
\begin{eqnarray*}
\sup_{f\leq t}\left(\frac{\arrowvert h\arrowvert}{R}\right)=\sup_{f=t}\left(\frac{\arrowvert h\arrowvert}{R}\right).
\end{eqnarray*}
Therefore, by assumption on the asymptotical behaviour of the scalar curvature, we get $\sup_{f=t}\left(\frac{\arrowvert h\arrowvert}{R}\right)\rightarrow 0$ as $t$ goes to $+\infty$. Hence $h=0$.
\end{proof}

Now comes the proof of theorem \ref{lic-op-dis-Bochner} :

\begin{proof}[Proof of theorem \ref{lic-op-dis-Bochner}]
Let $h$ be an eigenvector in $L^2_f$ of $L$ associated to an eigenvalue $\lambda\leq 0$. Then,
\begin{eqnarray*}
\Delta_f\arrowvert h\arrowvert^2=2\arrowvert\nabla h\arrowvert^2-2\lambda\arrowvert h\arrowvert^2-4<\Rm(g)\ast h,h>.
\end{eqnarray*}
By integrating this inequality, and by using the Kato inequality and the bound on the spectrum of the laplacian acting on functions given by proposition \ref{Mun-Wan-Har-Egs}, one gets,
\begin{eqnarray*}
\int_M\left(\inf_MR+\frac{n}{2}-2\|\Rm(g)\|_{L^{\infty}}\right)\arrowvert h\arrowvert^2d\mu_f\leq\lambda\|h\|^2_{L_f^2}\leq 0.
\end{eqnarray*}
By the very assumption made on the curvature tensor, one gets $h=0$.

\end{proof}

\appendix
\section{Appendix}
\subsection{Identities on gradient Ricci solitons}

We recall some classical identities on expanding and steady gradient Ricci solitons. See for instance [Chap.$4$,\cite{Ben}] for a proof of the following lemmata. 

\begin{lemma}\label{sol-id-egs}
Let $(M,g,\nabla f)$ be a normalized expanding gradient Ricci soliton, i.e. $\int_{M}e^{-f}d\mu(g)=1$. Then
\begin{eqnarray*}
&&\Delta f=R+\frac{n}{2},\\
&&\arrowvert \nabla f\arrowvert^2+R=f+\mu(g),\\
&&2\Ric(\nabla f)+\nabla R=0,
\end{eqnarray*}
\end{lemma}
where $R=R(g)$ is the scalar curvature and $\mu(g)$ is called the entropy of the expanding gradient Ricci soliton. 
%Here, we assume the following convention : $\Rm(g)_{ijji}$ is the sectional curvature of the metric $g$ for the plane spanned by $\{i,j\}$.

\begin{lemma}\label{sol-id-sgs}
Let $(M,g,\nabla f)$ be a steady gradient Ricci soliton. Then,
\begin{eqnarray*}
&&\Delta f=R,\\
&&\arrowvert \nabla f\arrowvert^2+R=\lambda(g),\\
&&2\Ric(\nabla f)+\nabla R=0,
\end{eqnarray*}
\end{lemma}
where $\lambda(g)$ is the energy of the steady gradient Ricci soliton.
\begin{rk}
A steady gradient Ricci soliton is not normalized by its definition. Sometimes, we do so by normalizing the energy $\lambda(g)$ to be $1$.
\end{rk}
\subsection{Growth of the potential function}

\begin{lemma}\label{gro-pot-fct}
\begin{enumerate}
\item Let $(M,g,\nabla f)$ be a steady gradient Ricci soliton with nonnegative Ricci curvature such that $\lim_{+\infty}R=0$. Then the potential function $f$ is proper and satisfies 
\begin{eqnarray}
c_1r_p(x)+c_2\leq f(x)\leq c_3r_p(x)+c_4 \label{gro-pot-fct-sgs},
\end{eqnarray}
for any fixed point $p\in M$ and positive constants $(c_i)_i$ depending on $p$ and $g$.\\
\item Let $(M,g,\nabla f)$ be an expanding gradient Ricci soliton satisfying $\nabla^2f\geq cg$ for a positive constant $c$. Then $f$ is proper and satisfies 
\begin{eqnarray}
c_1r_p^2(x)+c_2\leq f(x)\leq c_3r_p^2(x)+c_4 \label{gro-pot-fct-egs},
\end{eqnarray}
for any fixed point $p\in M$ and positive constants $(c_i)_i$ depending on $p$ and $g$.\\

\end{enumerate}
\end{lemma}
The proof of (\ref{gro-pot-fct-sgs}) can be found in \cite{Car-Ni} and the proof of the growth of the potential function on expanding gradient Ricci solitons is proved at least in \cite{Che-Der}.

\subsection{Equivalence of flows}

\begin{lemma}\label{equ-flo}
The Ricci flow and the modified Ricci flow are equivalent in the following sense :

1) If $(M,g(t))_{t\in[0,T)}$ is a solution to (MRF) then $(M,\tilde{g}(t))_{t\in[0,\tilde{T})}$ is a solution to the Ricci flow where
$$\tilde{g}(t):=(1+t)\phi_t^{\ast}g(\ln(1+t)),$$ where $(\phi_t)_t$ is the one parameter subgroup of diffeomorphisms generated by the time-dependent vector field $-X^0/(1+t)$ and $\tilde{T}:=e^T-1$.\\

2) And vice-versa.  
\end{lemma}

\bibliographystyle{alpha}
\bibliography{bib-egs-stab}

\end{document}